\documentclass[11pt,twoside,reqno]{amsart}

\usepackage{microtype}
\usepackage{cite}
\usepackage[OT1]{fontenc}
\usepackage{type1cm}
\usepackage{amssymb}
\usepackage{enumitem}
\usepackage{comment}
\usepackage{xcolor}
\usepackage{graphicx}

\usepackage{geometry}
\usepackage{hyperref}
\hypersetup{
  colorlinks=true,
  linkcolor=black,
  anchorcolor=black,
  citecolor=black,
  filecolor=black,
  menucolor=red,
  runcolor=black,
  urlcolor=black,
}

\numberwithin{equation}{section}

\theoremstyle{plain}
\newtheorem{theorem}{Theorem}[section]

\newtheorem{proposition}[theorem]{Proposition}

\newtheorem{lemma}[theorem]{Lemma}

\theoremstyle{remark}

\theoremstyle{definition}

\newcommand{\LL}{\mathcal{L}}

\newcommand{\MM}{\mathcal{M}}

\newcommand{\CC}{\mathcal{C}}

\newcommand{\R}{\mathbb{R}}

\newcommand{\N}{\mathbb{N}}
\newcommand{\hhh}{\mathtt{h}}
\newcommand{\iii}{\mathtt{i}}
\newcommand{\jjj}{\mathtt{j}}
\newcommand{\kkk}{\mathtt{k}}

\newcommand{\dd}{\,\mathrm{d}}
\newcommand{\ttt}{\mathbf{t}}
\newcommand{\www}{\mathbf{w}}
\newcommand{\vvv}{\mathbf{v}}

\renewcommand{\ge}{\geqslant}
\renewcommand{\le}{\leqslant}
\renewcommand{\geq}{\geqslant}
\renewcommand{\leq}{\leqslant}

\DeclareMathOperator{\dimloc}{dim_{loc}}

\DeclareMathOperator{\udimm}{\overline{dim}_M}

\DeclareMathOperator{\dimh}{dim_H}

\DeclareMathOperator{\ldimh}{\underline{dim}_H}

\DeclareMathOperator{\udimp}{\overline{dim}_p}

\DeclareMathOperator{\diml}{dim_L}

\DeclareMathOperator{\proj}{proj}

\renewcommand{\atop}[2]{\genfrac{}{}{0pt}{}{#1}{#2}}

\newcounter{nameOfYourChoice}

\allowdisplaybreaks

\begin{document}

\title[Dimension of non-conformal attractors]{Dimension of planar non-conformal attractors with triangular derivative matrices}

\author{Bal\'azs B\'ar\'any}
\address[Bal\'azs B\'ar\'any]
{Department of Stochastics \\
	Institute of Mathematics \\
	Budapest University of Technology and Economics \\
	M\H{u}egyetem rkp. 3 \\
	H-1111 Budapest,
	Hungary}
\email{barany.balazs@ttk.bme.hu}

\author{Antti K\"aenm\"aki}
\address[Antti K\"aenm\"aki]
         {Alfr\'ed R\'enyi Institute of Mathematics \\
          Hungarian Academy of Sciences \\ 
          Budapest \\ 
          Hungary}
\email{antti.kaenmaki@gmail.com}

\thanks{B. B\'ar\'any was supported by the grants NKFI FK134251, K142169, and the grant NKFI KKP144059 ``Fractal geometry and applications''.}
\date{\today}

\begin{abstract}
  We study the dimension of the attractor and quasi-Bernoulli measures of parametrized families of iterated function systems of non-conformal and non-affine maps. We introduce a transversality condition under which, relying on a weak Ledrappier-Young formula, we show that the dimensions equal to the root of the subadditive pressure and the Lyapunov dimension, respectively, for almost every choice of parameters. We also exhibit concrete examples satisfying the transversality condition with respect to the translation parameters.
\end{abstract}

\maketitle

\tableofcontents

\section{Introduction}

The dimension theory of iterated function systems is a rapidly developing branch of the geometric measure theory. There are still countless open questions. One of the main problems is to determine the dimension of the attractor of typical non-conformal and non-affine iterated function systems. An \emph{iterated function system} is a finite tuple $\Phi=(F_i)_{i=1}^N$ of $C^1$ contractions acting on $\R^d$. Hutchinson \cite{Hutchinson1981} showed that there exists a unique non-empty compact set $X$, called the \emph{attractor} of the $\Phi$, such that $X=\bigcup_{i=1}^NF_i(X)$. In this generality, the question of the dimension of the set $X$ is wide open.

In the special case, when the maps $F_i$ are similarities of $\R^d$, we call the attractor $X$ \emph{self-similar}, and if the \emph{strong separation condition} holds, i.e., $F_i(X) \cap F_j(X) = \emptyset$ whenever $i \ne j$, the dimension of $X$ can be easily calculated and expressed by using only the contraction ratios; see Hutchinson \cite{Hutchinson1981}. The situation is much harder to deal with when there is no separation condition. The first result considering this problem successfully was by Pollicott and Simon \cite{PoSi}. In their study, they introduced and used a method which is nowadays called the \emph{transversality method}. This method was later generalized, for example, by Solomyak \cite{Solomyak1998} for Bernoulli convolutions and by Simon, Solomyak, and Urba\'nski~\cite{SSU01,SSU01b} for iterated function systems formed by $C^{1+\alpha}$ maps on $\R$. Building on these ideas, many breakthrough results have appeared in the dimension theory of self-similar sets over the past decade; for example, see Hochman~\cite{Hochman2014,hochman2017selfsimilar}, Shmerkin~\cite{Shmerkin2014}, and Varj\'u~\cite{Varju19}.

Although addressing the dimension theory of iterated function systems formed by $C^{1+\alpha}$ maps on $\R$ is relatively easy, the higher dimensional analog is difficult as $C^{1+\alpha}$ maps on $\R^d$ are not necessarily conformal. In the non-conformal case, when the maps $F_i$ are invertible affine transformations, Falconer~\cite{Falconer1988} introduced an upper bound, called the \emph{affinity dimension}, for the dimension of the attractor $X$ which in this case is called the \emph{self-affine set}. He also proved that this bound is achieved for Lebesgue-almost every choice of the translation vectors of the affine maps $F_i$ having fixed linear parts with sufficiently small contraction ratios. A sharp bound for the contraction ratios was found by Solomyak~\cite{Solomyak1998} and the result was extended to measures by Jordan, Pollicott, and Simon~\cite{JordanPollicottSimon2007}. In all the mentioned works, the main tool was an affine version of the transversality method. Recently, the dimension of self-affine sets has been calculated in a deterministic setting under certain separation conditions in $\R^2$ and $\R^3$; see B\'ar\'any, Hochman, and Rapaport~\cite{BHR}, Hochman and Rapaport~\cite{HochmanRapaport2021} and Rapaport~\cite{rapaport2022self}. One of the main tools used in these results is the Ledrappier-Young formula introduced by Ledrappier and Young \cite{LedrappierYoung1,LedrappierYoung2} for $C^2$ diffeomorphisms on $C^\infty$ compact manifolds. The formula was later established for self-affine systems by Feng and Hu~\cite{FengHu2009}, B\'ar\'any and K\"aenm\"aki~\cite{BaranyKaenmaki2017}, and Feng~\cite{Feng23}. It connects the dimension of an ergodic measure supported on the self-affine set with the dimension of its projections along sufficient ``strongly contracting'' directions.

Surprisingly few results are known on the dimension theory of general non-conformal and non-affine systems. One example of such an object is the graph of the Weierstrass function. Let $\varphi\colon\R\to\R$ be a $1$-periodic $C^1$ function, $b\geq2$ an integer, and $b^{-1}<\lambda<1$. The \emph{Weierstrass function} is then $W(x)=\sum_{n=0}^\infty\lambda^n\varphi(b^nx)$. It is well-known that $W$ is a nowhere differentiable continuous function. It is also easy to see that the graph $\{(x,W(x)):x\in[0,1]\}$ is the attractor of the iterated function system $(F_i)_{i\in\{0,\ldots,b-1\}}$, where $F_i\colon\R^2\to\R^2$, $F_i(x,y)=(\frac{x+i}{b},\lambda y+\varphi(\frac{x+i}{b}))$. Ledrappier \cite{Ledrappier92} showed the Ledrappier-Young formula for such a system, and using that Bara\'nski, B\'ar\'any, and Romanowska \cite{BBR} studied the classical case $\varphi(x)=\cos(2\pi x)$ and verified Mandelbrot's conjecture that the dimension equals to $2+\frac{\log b}{\log\lambda}$ on a region of parameters. This result was later extended to the general case by Shen \cite{Shen}. Both results relied on a suitable transversality method and the task was to verify the assumptions required by Ledrappier~\cite{Ledrappier92}. Recently, Ren and Shen \cite{RenShen}, by relying on methods more resembling the methods used in \cite{BHR}, proved a nice dichotomy for the Weierstrass functions: for each analytic $\varphi$, either $W$ is analytic or the graph has dimension $2+\frac{\log b}{\log\lambda}$.

The general case of non-conformal and non-affine systems is far from being well understood. Falconer~\cite{Falconer94} found an upper bound in terms of a sub-additive pressure function for the upper Minkowski dimension of the attractor of an iterated function system formed by $C^2$-maps satisfying a technical condition called $1$-bunching. This was extended by Zhang~\cite{Zhang97} by showing that the sub-additive pressure gives an upper bound for the Hausdorff dimension of the attractor of an iterated function system formed by $C^1$ maps without the $1$-bunched condition. Only very recently, it was shown by Feng and Simon~\cite{FengSimon1} that this formula bounds also the upper Minkowski dimension in the case of $C^1$ maps. So far there have been only a very few cases when the dimension actually equals the upper bound given by the sub-additive pressure, and usually these cases deal with the Minkowski dimension under some special conditions; for example, see Falconer~\cite{Falconer94} and Barreira~\cite{Barreira96}. Feng and Simon \cite{FengSimon2} introduced a non-conformal and non-affine version of the transversality condition under which the Hausdorff and Minkowski dimensions equal the value given by the sub-additive pressure for almost every choice of parameters. They verified the result for certain iterated function systems where the functions have lower-triangular derivative matrices with a common strong stable direction. Jurga and Lee \cite{JurgaLeeLY} recently proved the Ledrappier-Young formula for such systems. For further developments, see the recent survey of Feng and Simon \cite{FengSimonsurvey}.

This paper is devoted to complementing the result of Feng and Simon \cite{FengSimon2}. We also restrict ourselves to functions with lower-triangular derivative matrices and further, we only consider the planar case. The Ledrappier-Young formula has been crucial in the development of the self-affine theory and our plan is to use it also here. Furthermore, we introduce a variant of the transversality condition which depends on the Ledrappier-Young formula. The difference with the already studied lower-triangular case is that instead of having a common strong stable direction, the maps now have a common weak stable direction. As the projections one has to use in this case come from certain ordinary differential equations and are non-linear, this setting cannot be studied by simply applying the existing methods from self-affine sets. For example, Jurga and Lee \cite{JurgaLeeLY} assumed the lower-triangular derivative matrices to have a common strong stable direction, so they were able to use orthogonal projections in their Ledrappier-Young formula.

In our main result, Theorem \ref{thm:meta}, we prove, by using a weak Ledrappier-Young formula we formulate and a transversality condition for the non-linear projections we introduce, that the Hausdorff and Minkowski dimension of the attractor of a general parametrized iterated function system with lower-triangular derivative matrices equals the value given by the sub-additive pressure for almost every choice of parameters. In Theorems \ref{thm:mainex1} and \ref{thm:mainex1b}, we exhibit two concrete classes of non-conformal and non-affine iterated function systems parametrized by their translation vectors for which the main theorem applies.

\subsection{Dimension estimates} \label{sec:dimension-estimates}
Let us now define the planar non-conformal iterated function systems we intend to study in details. Let $\Phi = (F_i)_{i \in \{1,\ldots,N\}}$, where $F_i \colon \R^2 \to \R^2$, $F_i(x,y) = (f_i(x),g_i(x,y))$, be an iterated function system such that
\begin{enumerate}[label=(G\arabic*)]
	\item\label{it:condLY1} $F_i\in C^{2}([0,1]^2)$ and $F_i([0,1]^2)\subseteq[0,1]^2$ for all $i\in\{1,\ldots,N\}$,
	\item\label{it:condLY2} there exist $0<\tau<\rho<1$ such that
  $$
  \tau < |f_i'(x)| < |(g_i)'_y(x,y)| < \rho
  $$
  for all $(x,y)\in[0,1]^2$ and $i\in\{1,\ldots,N\}$.
\end{enumerate}
We denote the attractor of $\Phi$ by $X$. Let us remark that the usually studied class of functions in the theory of iterated function systems is $C^{1+\alpha}$. However, in order to verify the transversality condition, we assume in \ref{eq:cond1} the partial derivatives of the second coordinate to be smooth. Therefore, for simplicity, we assume the maps to be in $C^2$. Note that the first coordinates of the maps in $\Phi$ form a $C^2$-conformal IFS on $[0,1]$. Let us denote this IFS by $\phi = (f_i)_{i\in\{1,\ldots,N\}}$.

We will next recall some standard notation. Let us denote by $\Sigma=\{1,\ldots,N\}^\N$ the \emph{symbolic space}, i.e., the set of all infinite words formed by the symbols $\{1,\ldots,N\}$. The set of $n$-length words is denoted by $\Sigma_n=\{1,\ldots,N\}^n$ and the set of finite length words by $\Sigma_*=\bigcup_{n=0}^\infty\Sigma_n$. We use the convention that $\Sigma_0=\{\varnothing\}$, where $\varnothing$ is the empty word, i.e., the identity element of the free monoid $\Sigma_*$. The length of a word $\iii\in\Sigma_*$ is denoted by $|\iii|$. For $\iii=i_1i_2\cdots\in\Sigma$, let $\iii|_n$ be the first $n$ coordinates of $\iii$, i.e., $\iii|_n=i_1\cdots i_n$. We use the convention that $\iii|_0=\varnothing$. For any two words $\iii,\jjj\in\Sigma\cup\Sigma_*$, let $|\iii\wedge\jjj|=\max\{k\geq0:\iii|_{k}=\jjj_{k}\}$ be the length of the common part and let $\iii\wedge\jjj=\iii|_{|\iii\wedge\jjj|}=\jjj|_{|\iii\wedge\jjj|}$ be the common part of $\iii$ and $\jjj$. The concatenation of $\iii\in\Sigma_*$ and $\jjj\in\Sigma\cup\Sigma_*$ is denoted by $\iii\jjj$. For a finite word $\iii=i_1\cdots i_n\in\Sigma_*$, let us denote by $\overleftarrow{\iii}=i_n\cdots i_1$ the finite word in reversed order. We also write $F_\iii=F_{i_1}\circ\cdots\circ F_{i_n}$ and denote the first and second coordinates of $F_\iii$ by $f_\iii$ and $g_\iii$, respectively. We use the convention that $F_\varnothing(x,y)=(x,y)$, $f_\varnothing(x)=x$ and $g_\varnothing(x,y)=y$. Throughout the paper, we will denote the derivative of the maps $g_i$ (and respectively their iterates $g_\iii$) with respect to the first coordinate by $(g_i)'_x$ and with respect to the second coordinate by $(g_i)_y'$ (and respectively $(g_\iii)'_x$ and $(g_\iii)_y'$ for higher iterates). Since the first coordinate of the map $F_i$ depend only on the first variable, we denote the derivative of $f_i$ (and $f_\iii$) by $f_i'$ (and by $f_\iii'$, respectively).

Let $\pi\colon\Sigma\to X$ be the \emph{canonical projection} defined by
\begin{equation} \label{eq:natproj}
\pi(\iii)=\lim_{n\to\infty}F_{\iii|_n}(x,y)
\end{equation}
for all $\iii \in \Sigma$, where the value of $\pi(\iii)$ is independent of the choice of $(x,y)\in[0,1]^2$. Let us denote the first and second coordinates of $\pi(\iii)$ by $\pi^1(\iii)$ and $\pi^2(\iii)$, respectively. Observe that $\pi^1$ is the canonical projection of $\phi$, i.e., $\pi^1(\iii)=\lim_{n\to\infty}f_{\iii|_n}(x)$ for all $\iii \in \Sigma$ and $x \in [0,1]$.

For $\iii\in\Sigma_*$, let
\begin{equation} \label{eq:singvaluefunction}
\varphi^s(\iii,x,y)=
\begin{cases}|(g_\iii)'_y(x,y)|^s, &\text{if $0\leq s\leq1$},\\
	|(g_\iii)'_y(x,y)||(f_\iii)'(x)|^{s-1}, &\text{if $1< s\leq2$},\\
	|(g_\iii)'_y(x,y)(f_\iii)'(x)|^{s/2}, &\text{if $s>2$}.
\end{cases}
\end{equation}
We define the \emph{subadditive pressure function} $P\colon[0,\infty)\mapsto\R$ by setting
\begin{equation} \label{eq:subpressure}
P(s)=\lim_{n\to\infty}\frac{1}{n}\log\sum_{\iii\in\Sigma_n}\varphi^s(\iii,x,y)
\end{equation}
for all $s \ge 0$, where the limit exists and is independent of the choice of $(x,y)\in[0,1]^2$ by \cite[Theorem~3.1]{B2009}. This fact strongly relies on the bounded distortion property of the iterates of the maps: there exists $C>1$ such that, for each $\iii\in\Sigma_*$ and for every $(x,y),(x',y')\in[0,1]^2$, we have
\begin{equation} \label{eq:bd}
	C^{-1}\leq\frac{|f_\iii'(x)|}{|f_\iii'(x')|}\leq C \qquad\text{and}\qquad C^{-1}\leq\frac{|(g_{\iii})'_y(x,y)|}{|(g_{\iii})'_y(x',y')|}\leq C;
\end{equation}
see \cite[Proposition~2.1]{B2009}. Relying on \ref{it:condLY2}, it is easy to see that $s\mapsto P(s)$ is a continuous, strictly decreasing function such that $P(0)=\log N$ and $\lim_{s\to\infty}P(s)=-\infty$. Hence, there exists a unique root of the subadditive pressure function and we denote it by $s_0$. It follows from a recent general result of Feng and Simon \cite[Theorem 1.1]{FengSimon1} that
\begin{equation} \label{eq:upperboundforset}
	\udimm(X)\leq\min\{2,s_0\},
\end{equation}
where $\udimm$ is the upper Minkowski dimension of a given set.

We denote the collection of all Borel probability measures on $\Sigma$ by $\MM(\Sigma)$, and endow it with the weak$^*$ topology. The \emph{left shift} is a map $\sigma \colon \Sigma \to \Sigma$ defined by setting $\sigma\iii = \sigma(\iii) = i_2i_3\cdots$ for all $\iii = i_1i_2\cdots \in \Sigma$. We say that a measure $\mu \in \MM(\Sigma)$ is \emph{$\sigma$-invariant} if $\mu([\iii]) = \sum_{i=1}^N\mu([i\iii])$ for all $\iii \in \Sigma_*$ and \emph{ergodic} if $\mu(A)=0$ or $\mu(A)=1$ for every Borel set $A \subseteq \Sigma$ with $\sigma^{-1}(A)=A$. Write
\begin{equation*}
  \MM_\sigma(\Sigma) = \{\mu \in \MM(\Sigma) : \mu \text{ is $\sigma$-invariant}\}.
\end{equation*}
Recall that the \emph{Kolmogorov-Sinai entropy} of $\mu \in \MM(\Sigma)$ is
\begin{equation*}
  h(\mu) = -\lim_{n \to \infty} \frac{1}{n}\sum_{\iii \in \Sigma_n} \mu([\iii])\log\mu([\iii])
\end{equation*}
and the \emph{Lyapunov exponents} of $\mu$ are
\begin{equation} \label{eq:lyapexp}
\begin{split}
  \chi_1(\mu) &= -\int\log|(g_{i_1})'_y(\pi(\sigma\iii))|\dd\mu(\iii), \\ 
  \chi_2(\mu) &= -\int\log|f_{i_1}'(\pi^1(\sigma\iii))|\dd\mu(\iii).
\end{split}
\end{equation}
Feng and Simon \cite[Theorem 1.2]{FengSimon1} have shown that
\begin{equation}\label{eq:upperboundformeasure}
	\udimp(\pi_*\mu) \leq \min\biggl\{2,\frac{h(\mu)}{\chi_1(\mu)},1+\frac{h(\mu)-\chi_1(\mu)}{\chi_2(\mu)}\biggr\},
\end{equation}
where $\udimp$ is the upper packing dimension of a given measure. The expression on the right-hand side of the inequality above is called the \emph{Lyapunov dimension} of $\mu$ and we denote it by $\diml(\mu)$. To infer the dimension result for the attractor, it suffices to work with a smaller class of invariant measures. We say that $\mu \in \MM(\Sigma)$ is \emph{quasi-Bernoulli} if there exists a constant $C \ge 1$ such that
\begin{equation*}
  C^{-1}\mu([\iii])\mu([\jjj]) \le \mu([\iii\jjj]) \le C\mu([\iii])\mu([\jjj])
\end{equation*}
for all $\iii,\jjj \in \Sigma_*$. It is straightforward to see that every quasi-Bernoulli measure is equivalent to a $\sigma$-invariant quasi-Bernoulli measure which furthermore is ergodic.

\subsection{Two concrete examples} \label{sec:examples}
Let us now consider parametrized families of iterated function systems satisfying \ref{it:condLY1} and \ref{it:condLY2}. A natural way to parametrize such systems is via translation parameters $\ttt \in U$. To emphasize the dependence on the parameter $\ttt \in U$, we denote the root of the subadditive pressure \eqref{eq:subpressure} by $s_0(\ttt)$, the Lyapunov exponents of \eqref{eq:lyapexp} by $\chi_1(\mu,\ttt)$ and $\chi_2(\mu,\ttt)$, the Lyapunov dimension of \eqref{eq:upperboundformeasure} by $\diml(\mu,\ttt)$, the canonical projection of \eqref{eq:natproj} by $\pi_\ttt$, and the attractor by $X_\ttt = \pi_\ttt(\Sigma)$. Let $\ldimh$ be the lower Hausdorff dimension of a given measure.

To prove the desired dimension result for $X_\ttt$, it suffices to work with quasi-Bernoulli measures $\mu \in \MM_\sigma(\Sigma)$. The task is to show that $(\pi_\ttt)_*\mu$ satisfies
\begin{equation*}
  \ldimh((\pi_\ttt)_*\mu) \ge \diml(\mu,\ttt)
\end{equation*}
for almost all $\ttt\in U$. By \eqref{eq:upperboundformeasure}, the measure $(\pi_\ttt)_*\mu$ then satisfies $\ldimh((\pi_\ttt)_*\mu) = \udimp((\pi_\ttt)_*\mu)$ for almost all $\ttt \in U$. If a Borel measure $m$ satisfies such an equality, then we say that $m$ is \emph{exact-dimensional} and we denote the common value by $\dim(m)$. The equality also implies that the \emph{pointwise dimension} of $m$, $\dimloc(m,x)$, exists at almost every $x$, is almost everywhere constant, and the value equals $\dim(m)$. In other words,
\begin{equation*} 
  \dimloc(m,x) = \lim_{r \downarrow 0} \frac{\log m(B(x,r))}{\log r} = \dim(m)
\end{equation*}
for $m$-almost all $x \in \R^2$.

Our first example is a planar non-conformal iterated function system $\Phi^{\ttt} = (F_i^\ttt)_{i \in \{1,\ldots,N\}}$, where $F_i^\ttt \colon \R^2 \to \R^2$, $F_i^\ttt(x,y) = (f_i(x)+t_{i,1},g_i(x,y)+t_{i,2})$, are $C^2$ maps parametrized on an open and bounded set $U\subset\R^{2N}$ such that
\begin{figure}[t]
    \centering
    \includegraphics[width=0.8\linewidth]{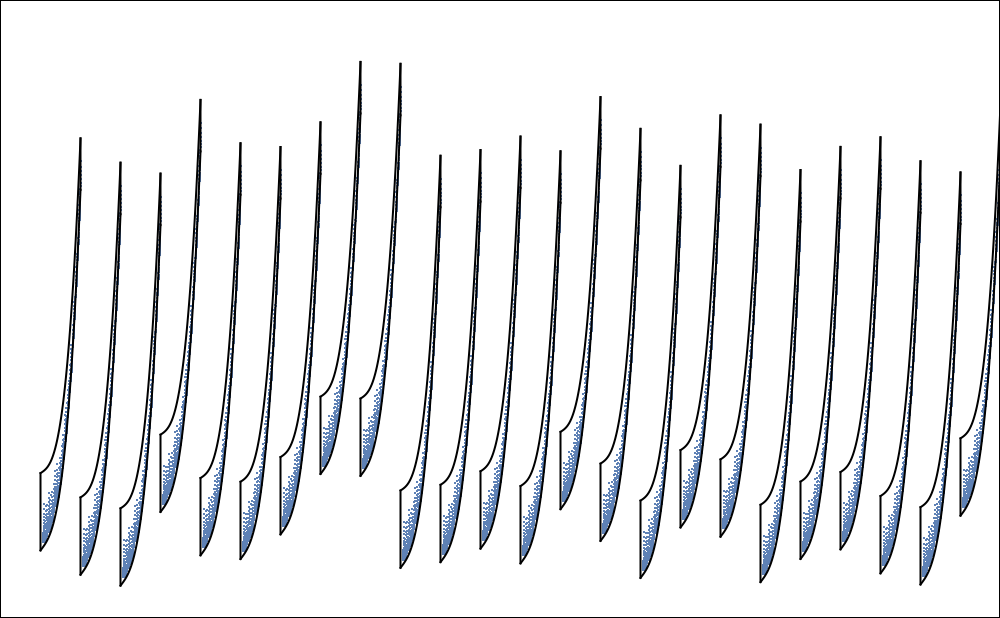}
    \caption{The attractor of an iterated function system $\Phi$ satisfying the assumptions \ref{it:1}--\ref{it:3} inside the first-level cylinders.}
    \label{fig:attractor}
\end{figure}
\begin{enumerate}[label=(A\arabic*)]
	\item\label{it:1} it holds that
  \begin{equation*}
    0 < \tau < f_i'(x) < (g_i)'_y(x,y) < \rho < \frac12
  \end{equation*}
  for all $(x,y)\in[0,1]^2$ and $i \in \{1,\ldots,N\}$,
	\item\label{it:2} it holds that
  \begin{equation*}
    (g_i)''_{xy}(x,y) \le 0 \le (g_i)'_x(x,y) \qquad\text{and}\qquad (g_j)''_{yy}(x,y) \ge 0
  \end{equation*}
  for all $(x,y)\in[0,1]^2$ and $i\in\{1,\ldots,N\}$,
	\item\label{it:3} $F_i^\ttt([0,1]^2)\subseteq[0,1]^2$ for all $i \in \{1,\ldots,N\}$ and $\ttt = ((t_{1,1},t_{1,2}),\ldots,(t_{N,1},t_{N,2})) \in U$.
\end{enumerate}
For example, $\Phi = (F_i+(t_{i,1},t_{i,2}))_{i=1}^{24}$, where
\[
F_i(x,y)=\biggl(\frac{x+i}{25},\frac{y^2+2y+1-xy+12x^3+2x}{24}\biggr),
\]
satisfies the assumptions \ref{it:1}--\ref{it:3} on a sufficiently small neighbourhood of the origin in $\R^{48}$; see Figure \ref{fig:attractor}. We also remark that the assumption \ref{it:2} can be replaced by
\begin{enumerate}
	\item[(A2')] $(g_i)''_{xy}(x,y)\leq0$, $(g_i)'_x(x,y)\leq0$, and $(g_i)''_{yy}(x,y)\leq0$ for all $(x,y)\in[0,1]^2$ and $i \in \{1,\ldots,N\}$.
\end{enumerate}
Let $\LL^d$ be the Lebesgue measure on $\R^d$ and $\dimh$ be the Hausdorff dimension of a given set.

\begin{theorem}\label{thm:mainex1}
	Suppose that $\Phi^{\ttt}$ satisfies the assumptions \ref{it:1}--\ref{it:3}. If $\mu \in \MM_\sigma(\Sigma)$ is quasi-Bernoulli, then
  \begin{equation*}
    \dim((\pi_\ttt)_*\mu) = \diml(\mu,\ttt)
  \end{equation*}
  for $\LL^{2N}$-almost all $\ttt\in U$ and $(\pi_\ttt)_*\mu\ll\LL^2$ for $\LL^{2N}$-almost all $\ttt \in U$ for which $h(\mu)>\chi_1(\mu,\ttt)+\chi_2(\mu,\ttt)\}$. Furthermore,
  \begin{equation*}
    \dimh(X_\ttt) = \udimm(X_\ttt) = \min\{2,s_0(\ttt)\}
  \end{equation*}
  for $\LL^{2N}$-almost all $\ttt\in U$ and $\LL^2(X_\ttt)>0$ for $\LL^{2N}$-almost all $\ttt \in U$ with $s_0(\ttt)>2$.
\end{theorem}

The second parametrized family is a planar non-conformal iterated function system $\Phi^{\ttt} = (F_i^\ttt)_{i \in \{1,\ldots,N\}}$, where $F_i^\ttt \colon \R^2 \to \R^2$, $F_i^\ttt(x,y) = (f_i(x)+t_{i,1},g_i(x,y)+t_{i,2})$, are $C^2$ maps parametrized on an open and bounded set $U\subset\R^{2N}$ such that
\begin{enumerate}[label=(B\arabic*)]
  \item\label{it:1b} it holds that
  \begin{equation*}
    0 < 4\tau < 4|f_i'(x)| < |(g_i)'_y(x,y)| < \frac14
  \end{equation*}
  for all $(x,y)\in[0,1]^2$ and $i \in \{1,\ldots,N\}$,
  \item\label{it:2b} it holds that
  \begin{equation*}
    \max_{\atop{(x,y)\in[0,1]^2}{i\in\{1,\ldots,N\}}}\frac{|(g_i)''_{xy}(x,y)|}{|(g_i)_y'(x,y)|}\leq\frac13
  \end{equation*}
  and
  \begin{equation*}
    \max_{\atop{(x,y)\in[0,1]^2}{i\in\{1,\ldots,N\}}}\frac{|(g_i)'_x(x,y)|}{|(g_i)'_y(x,y)|}\cdot\max_{\atop{(x,y)\in[0,1]^2}{j\in\{1,\ldots,N\}}}\frac{| (g_j)_{yy}''(x,y)|}{|(g_j)_y'(x,y)|}\leq\frac13
  \end{equation*}
  for all $(x,y)\in[0,1]^2$,
  \item\label{it:3b} $F_i^\ttt([0,1]^2)\subseteq[0,1]^2$ for all $i \in \{1,\ldots,N\}$ and $\ttt = ((t_{1,1},t_{1,2}),\ldots,(t_{N,1},t_{N,2})) \in U$.
\end{enumerate}
For example, $\Phi = (F_i+(t_{i,1},t_{i,2}))_{i=1}^{13}$, where
\begin{equation*}
  F_i(x,y)=\biggl(e^{(x-i)/25},\frac{y}{5} e^{(x-i)/25}+\frac{\cos(x)}{2}+\frac{i-6}{25}\biggr),
\end{equation*}
satisfies the assumptions \ref{it:1b}--\ref{it:3b} on a sufficiently small neighbourhood of the origin in $\R^{26}$. The result for this family is precisely the same.

\begin{theorem}\label{thm:mainex1b}
  Suppose that $\Phi^{\ttt}$ satisfies the assumptions \ref{it:1b}--\ref{it:3b}. If $\mu \in \MM_\sigma(\Sigma)$ is quasi-Bernoulli, then
  \begin{equation*}
    \dim((\pi_\ttt)_*\mu) = \diml(\mu,\ttt)
  \end{equation*}
  for $\LL^{2N}$-almost all $\ttt\in U$ and $(\pi_\ttt)_*\mu\ll\LL^2$ for $\LL^{2N}$-almost all $\ttt \in U$ for which $h(\mu)>\chi_1(\mu,\ttt)+\chi_2(\mu,\ttt)\}$. Furthermore,
  \begin{equation*}
    \dimh(X_\ttt) = \udimm(X_\ttt) = \min\{2,s_0(\ttt)\}
  \end{equation*}
  for $\LL^{2N}$-almost all $\ttt\in U$ and $\LL^2(X_\ttt)>0$ for $\LL^{2N}$-almost all $\ttt \in U$ with $s_0(\ttt)>2$.
\end{theorem}

Theorems~\ref{thm:mainex1} and \ref{thm:mainex1b} are consequences of a more general theorem, Theorem \ref{thm:meta}, presented in Section~\ref{sec:main}. A pivotal assumption in the theorem is a suitable transversality condition for general planar non-conformal iterated function systems which we introduce to guarantee that there is no dimension drop for almost every parameter. In Section \ref{sec:verify-examples}, we verify that the assumptions \ref{it:1}--\ref{it:3} and \ref{it:1b}--\ref{it:3b} imply this transversality condition.

\section{Weak Ledrappier-Young formula}

One of the main tools in the study of the dimension theory of iterated function systems is the Ledrappier-Young formula. In this section, we do not intend to generalize the formula in its strongest form. We only give a lower bound for the dimension of our planar invariant measure by its non-linear projections. Before stating this weak Ledrappier-Young formula, we go through some preliminaries. 

Let $\Phi = (F_i)_{i \in \{1,\ldots,N\}}$ be an IFS satisfying \ref{it:condLY1} and \ref{it:condLY2}. Let us define the slope of the strong-stable tangent bundle $u\colon\Sigma\times\R^2\to\R$ as follows: For each $\iii\in\Sigma$ and for every $(x,y)\in[0,1]^2$, let
\begin{equation}\label{eq:defofbundle}
u(\iii,x,y) = \sum_{k=1}^\infty\frac{-(g_{i_k})_x'(F_{\overleftarrow{\iii|_{k-1}}}(x,y))f_{\overleftarrow{\iii|_{k-1}}}'(x)}{(g_{\overleftarrow{\iii|_{k}}})'_y(x,y)}.
\end{equation}
It is easy to see that the series above is absolutely convergent by \ref{it:condLY2}. The importance of the mapping $u$ is that it forms an invariant bundle with respect to the derivatives, and the contraction of the map $F_{i_1}$ is the ``strongest'' along lines with slope $u(\iii,\,\cdot\,,\,\cdot\,)$. More precisely, by the definition \eqref{eq:defofbundle},
\begin{equation}\label{eq:bundleprop}
u(\iii,x,y)=-\frac{(g_{i_1})'_x(x,y)}{(g_{i_1})_y'(x,y)}+\frac{f_{i_1}'(x)}{(g_{i_1})_y'(x,y)}u(\sigma\iii,F_{i_1}(x,y))
\end{equation}
and hence,
$$
D_{(x,y)}F_{i_1}\begin{pmatrix}
	1 \\
	u(\iii,x,y)
\end{pmatrix}=f_{i_1}'(x)\begin{pmatrix}
	1\\
	u(\sigma\iii,F_{i_1}(x,y))
\end{pmatrix}.
$$
for all $\iii=i_1i_2\cdots\in\Sigma$ and $(x,y)\in[0,1]^2$.

Let us next define the strong-stable foliation corresponding to the strong-stable tangent bundle. In other words, for each $\iii\in\Sigma$ we partition $[0,1]\times\R$ into $C^1$-smooth curves. By our assumptions, each of these curves can be represented as a graph of a $C^1$ map $x\mapsto y(x)$. More precisely, for each $\iii\in\Sigma$ and for every $(x_0,y_0)\in[0,1]^2$ there exists an ordinary differential equation with a unique $C^1$-solution denoted by $x\mapsto y(\iii,(x_0,y_0),x)$ such that its value equals $y_0$ at $x=x_0$ and its derivative with respect to $x$ corresponds to the slope $u(\iii,x,y(\iii,(x_0,y_0),x))$, that is,
\begin{equation}\label{eq:diffeq}
\begin{split}
	(y)'_x(\iii,(x_0,y_0),x) &= u(\iii,x,y(\iii,(x_0,y_0),x)), \\
	y(\iii,(x_0,y_0),x_0) &= y_0,
\end{split}
\end{equation}
or, alternatively,
$$
  y(\iii,(x_0,y_0),x)=y_0+\int_{x_0}^xu(\iii,z,y(\iii,(x_0,y_0),z))\dd z.
$$
Hence, for each $\iii\in\Sigma$ we can define the strong-stable foliation $\xi_\iii$ as
$$
  \xi_\iii(x_0,y_0) = \{(x,y(\iii,(x_0,y_0),x)):x\in[0,1]\}.
$$
The importance of this foliation is that it is invariant with respect to $\Phi$.

\begin{lemma}\label{lem:invariance}
	For each $\iii\in\Sigma$ and for every $(x_0,y_0)\in[0,1]^2$ it holds that
	\begin{equation*}
		g_{i_1}(x,y(\iii,(x_0,y_0),x)) = y(\sigma\iii,F_{i_1}(x_0,y_0),f_{i_1}(x))
	\end{equation*}
  for all $x\in[0,1]$.
\end{lemma}

\begin{proof}
Let us consider the derivative of the map $x\mapsto g_{i_1}(x,y(\iii,(x_0,y_0),x))$. By \eqref{eq:bundleprop}, we have
	\[
	\begin{split}
		\frac{\mathrm{d}}{\mathrm{d}x}g_{i_1}(x,y(\iii,(x_0,y_0),x)) &= (g_{i_1})'_x(x,y(\iii,(x_0,y_0),x)) \\
		&\qquad\qquad+ (g_{i_1})_y'(x,y(\iii,(x_0,y_0),x))u(\iii,x,y(\iii,(x_0,y_0),x)) \\
		&= u(\sigma\iii,F_{i_1}(x,y(\iii,(x_0,y_0),x)))f_{i_1}'(x) \\
		&= u(\sigma\iii,f_{i_1}(x),g_{i_1}(x,y(\iii,(x_0,y_0),x)))f_{i_1}'(x).
	\end{split}
	\]
	On the other hand, the derivative of the map $x\mapsto y(\sigma\iii,F_{i_1}(x_0,y_0),f_{i_1}(x))$ is
	\[\begin{split}
	\frac{\mathrm{d}}{\mathrm{d}x}y(\sigma\iii,F_{i_1}(x_0,y_0),f_{i_1}(x))&=(y)'_x(\sigma\iii,F_{i_1}(x_0,y_0),f_{i_1}(x)))f_{i_1}'(x)\\
	&=u(\sigma\iii,f_{i_1}(x),y(\sigma\iii,F_{i_1}(x_0,y_0),f_{i_1}(x)))f_{i_1}'(x).
	\end{split}\]
	Since $g_{i_1}(x_0,y(\iii,(x_0,y_0),x_0))=g_{i_1}(x_0,y_0)=y(\sigma\iii,F_{i_1}(x_0,y_0),f_{i_1}(x_0))$ by \eqref{eq:diffeq}, the claim follows from the classical Picard-Lindel\"of Theorem; see, for example, \cite[Theorem~3.2]{ODE}.
\end{proof}

For each $\iii\in\Sigma$, let us define a non-linear projection $\proj_\iii \colon \Sigma \to \R$ by setting
\begin{equation}\label{eq:nonlinproj}
\proj_\iii(\jjj)=y(\iii,\pi(\jjj),0)
\end{equation}
for all $\jjj \in \Sigma$. For the measurable partition $\{\pi^{-1}(x,y)\}_{(x,y)\in[0,1]^2}$ of $\Sigma$, let $\{\mu_{\iii}^{\pi}\}_{\iii\in\Sigma}$ be the family of conditional measures supported on the partition element $\pi^{-1}(\pi(\iii))$. Note that such a family of measures is defined uniquely up to a zero measure set; see, for example, Simmons \cite{Simmons}. Finally, for each $\mu \in \MM_\sigma(\Sigma)$, let us define the reversed measure $\overleftarrow{\mu} \in \MM_\sigma(\Sigma)$ by setting $\overleftarrow{\mu}([\iii])=\mu([\overleftarrow{\iii}])$ for all $\iii\in\Sigma_*$ and extending in usual manner. We are now ready to state and prove the weak Ledrappier-Young formula.

\begin{theorem}[Weak Ledrappier-Young formula] \label{thm:LY}
Suppose that $\Phi$ satisfies the assumptions \ref{it:condLY1} and \ref{it:condLY2}. If $\mu \in \MM_\sigma(\Sigma)$ is quasi-Bernoulli, then
\begin{equation}\label{eq:LY}
	\ldimh(\pi_*\mu) \ge \frac{h(\mu)-H(\mu)}{\chi_2(\mu)} + \biggl(1-\frac{\chi_1(\mu)}{\chi_2(\mu)}\biggr)\dim((\proj_\iii)_*\mu)
\end{equation}
for $\overleftarrow{\mu}$-almost all $\iii$, where $H(\mu)=-\int\log\mu_\jjj^{\pi}([j_1])\dd\mu(\jjj)$.
\end{theorem}

According to our best knowledge, Theorem~\ref{thm:LY} has not been stated earlier in this context and it is original contribution within this work. Compared to the usual Ledrappier-Young formula, it does not demonstrate the exact-dimensionality of the measure and only gives a lower bound for the dimension of the measure by means of the projection, entropy, and Lyapunov exponents. Nevertheless, this lower bound is still sufficient for studying the dimension and the proof is a slight modification of the proofs of \cite[Theorem~2.6]{BaranyKaenmaki2017} and \cite[Proposition~2]{Ledrappier92} for quasi-Bernoulli measures on dominated self-affine systems. For the convenience of the reader, we present below the necessary changes following \cite[Sections~6~and~7]{BaranyKaenmaki2017}. We strongly believe that the upper bound holds as well.

\begin{proof}[Sketch of the proof of Theorem~\ref{thm:LY}]
Let $\rho>0$ be such that $\rho<1/N$ and choose real numbers $t_1,\ldots,t_N$ such that the IFS $\{x\mapsto\rho x+t_i\}_{i=1}^N$ satisfies the strong separation condition and acts on $[0,1]$. Write $\hat\Phi = (\hat{F}_i)_{i \in \{1,\ldots,N\}}$, where $\hat{F}_i \colon \R^3 \to \R^3$, $\hat{F}_i(x,y,z) = (f_i(x),g_i(x,y),\rho z+t_i)$, lift the maps $F_i$ into $\R^3$ and make the system satisfy the strong separation condition. We denote the attractor of $\hat{\Phi}$ by $\hat{X}$ and the canonical projection by $\Pi\colon\Sigma\to\hat{X}$.

Define a dynamical system $T\colon[0,1]^3\times\Sigma\to[0,1]^3\times\Sigma$ by setting
$$
T(\mathbf{x},\iii)=(\hat{F}_{i_1}(\mathbf{x}),\sigma\iii).
$$
Since $\hat{\Phi}$ satisfies the strong separation condition, the inverse $T^{-1}$ is well-defined on $\hat{X}\times\Sigma$. Let $P_2$ denote the orthogonal projection from $[0,1]^3$ onto the first two coordinates. Note that we have $P_2\circ\Pi=\pi$. Let us define three $T^{-1}$-invariant foliations $\xi^0$, $\xi^1$, and $\xi^2$ of $\hat{X}\times\Sigma$ by setting, for each $(\mathbf{x},\iii)\in\hat{X}\times\Sigma$,
\begin{align*}
  \xi^0(\mathbf{x},\iii) &= \hat{X}\times\{\iii\}, \\
  \xi^1(\mathbf{x},\iii) &= \{\mathbf{y}\in\hat{X} : y(\iii,P_2(\mathbf{x}),0)=y(\iii,P_2(\mathbf{y}),0)\}\times\{\iii\}, \\
  \xi^2(\mathbf{x},\iii) &= \{\mathbf{y}\in\hat{X} : P_2(\mathbf{x})=P_2(\mathbf{y})\}\times\{\iii\}.
\end{align*}
Furthermore, let $\mathcal{P}(\mathbf{x},\iii) = \hat{F}_{j}(\hat{X}) \times \Sigma$, where $j$ is the unique symbol such that $\mathbf{x}\in\hat{F}_j(\hat{X})$. It follows that the foliations $\xi^i$, $i \in \{0,1,2\}$, are invariant, i.e.,
\begin{equation}\label{eq:invfoliation}
  \xi^i\vee\mathcal{P}=T\xi^i,
\end{equation}
where $(\xi^i\vee\mathcal{P})(\mathbf{x},\iii)=\xi^i(\mathbf{x},\iii)\cap\mathcal{P}(\mathbf{x},\iii)$ is the common refinement of the partitions $\xi^i$ and $\mathcal{P}$ and $T\xi^i$ is the partition $(T\xi^i)(\mathbf{x},\iii)=T(\xi(T^{-1}(\mathbf{x},\iii)))$. The cases $i \in \{0,2\}$ are straightforward and the case $i=1$ follows by Lemma~\ref{lem:invariance}. Let us denote by $\mathcal{P}^n$ the refinement of the partition $\mathcal{P}$ along the dynamics of $T^{-1}$, i.e.,
$$
  \mathcal{P}^n(\mathbf{x},\iii)=\mathcal{P}(\mathbf{x},\iii)\cap T(\mathcal{P}(T^{-1}(\mathbf{x},\iii)))\cap\cdots\cap T^{n-1}(\mathcal{P}(T^{-(n-1)}(\mathbf{x},\iii))).
$$
We use the convention that $\mathcal{P}^0(\mathbf{x},\iii)=\hat{X}\times\Sigma$. Finally, let us define the transversal balls by setting
\begin{align*}
	B_2^T((\mathbf{x},\iii),\delta) &= \{(\mathbf{y},\jjj)\in\xi^1(\mathbf{x},\iii):|P_2(\mathbf{x})-P_2(\mathbf{y})|<\delta\}, \\
	B_1^T((\mathbf{x},\iii),\delta) &= \{(\mathbf{y},\jjj)\in\xi^0(\mathbf{x},\iii):|y(\iii,P_2(\mathbf{x}),0)-y(\iii,P_2(\mathbf{y}),0)|<\delta\}.
\end{align*}
It is easy to see, by the bounded distortion \eqref{eq:bd} and the invariance of the foliations \eqref{eq:invfoliation}, that there exists $c>0$ such that
\begin{equation}\label{eq:containment}
\begin{split}
  B_i^T(T(\mathbf{x},\iii),c^{-1}\delta)\cap\mathcal{P}(T(\mathbf{x},\iii)) &\subseteq T(B_i^T((\mathbf{x},\iii),\delta)) \\ 
  &\subseteq B_i^T(T(\mathbf{x},\iii),c\delta)\cap\mathcal{P}(T(\mathbf{x},\iii))
\end{split}
\end{equation}
for all $(\mathbf{x},\iii)\in\hat{X}\times\Sigma$ and $\delta>0$.

Let $\mu \in \MM_\sigma(\Sigma)$ be a quasi-Bernoulli measure. Recall that $\mu$ is automatically ergodic. Furthermore, there exists a unique ergodic $T^{-1}$-invariant Borel probability measure $\nu$ such that $\nu$ is equivalent to the measure $\Pi_*\mu\times\overleftarrow{\mu}$ on $\hat{X}\times\Sigma$. To simplify notation, let us denote the measure $\pi_*\mu$ by $m$. Let $\nu^{\xi^i}_{(\mathbf{x},\iii)}$ be the family of conditional measures with respect to the foliation $\xi^i$; see \cite[Theorem~2.1 and Theorem~2.2]{Simmons}. It follows that $\nu_{(\mathbf{x},\iii)}^{\xi^0}$ is equivalent to $\Pi_*\mu$ and $(P_2)_*\nu_{(\mathbf{x},\iii)}^{\xi^0}$ is equivalent to $m$ for $\nu$-almost all $(\mathbf{x},\iii)$. For each $\iii\in\Sigma$, let $\mu_{(\jjj,\iii)}^{\proj}$ be the family of conditional measures with respect to the planar foliation $\{\proj_{\iii}^{-1}(\proj_{\iii}(\jjj))\}$. It is easy to see that $(P_2)_*\nu^{\xi^1}_{\Pi(\jjj),\iii}$ is equivalent to $m_{(\pi(\jjj),\iii)}^{\proj}=\pi_*\mu_{(\jjj,\iii)}^{\proj}$ for $\mu$-almost all $\jjj$.

By following the argument of \cite[Lemma~6.3]{BaranyKaenmaki2017}, one can show, by relying on \eqref{eq:containment} and the uniqueness of the conditional measures, that
\begin{equation}\label{eq:invmeas}
	\nu^{\xi^i}_{T^{-1}(\mathbf{x},\iii)}=(T^{-1})_*\nu^{\xi^i\vee\mathcal{P}}_{(\mathbf{x},\iii)}
\end{equation}
for $\nu$-almost all $(\mathbf{x},\iii)$. By induction, one can also show that
\begin{equation}\label{eq:toharmonic}
	\frac{\nu_{(\mathbf{x},\iii)}^{\xi^i}(B_{i+1}^T\big((\mathbf{x},\iii),\delta\big)\cap\mathcal{P}^n(\mathbf{x},\iii))}{\nu_{(\mathbf{x},\iii)}^{\xi^i}(\mathcal{P}^n(\mathbf{x},\iii))}=\nu_{T^{-n}(\mathbf{x},\iii)}^{\xi^i}(T^{-n}(B_{i+1}^T\big((\mathbf{x},\iii),\delta\big))).
\end{equation}
for all $n\in\N$. By the bounded distortion \eqref{eq:bd}, there exists $C>0$ such that for each $(\jjj,\iii)\in\Sigma\times\Sigma$ we have
\begin{align*}
	B_1^T(T^{-n}(\Pi(\jjj),\iii),C^{-1}\delta\|(g_{\jjj|_n})'_y\|^{-1}) &\subseteq T^{-n}(B_1^T((\Pi(\jjj),\iii),\delta))\cap \xi^0(T^{-n}(\Pi(\jjj),\iii))\\
	&\subseteq B_1^T(T^{-n}(\Pi(\jjj),\iii),C\delta\|(g_{\jjj|_n})'_y\|^{-1})
\end{align*}
and
\begin{align}
		B_2^T(T^{-n}(\Pi(\jjj),\iii),C^{-1}\delta\|(f_{\jjj|_n})'\|^{-1})&\subseteq T^{-n}(B_2^T((\Pi(\jjj),\iii),\delta))\cap \xi^1(T^{-n}(\Pi(\jjj),\iii))\notag\\
		&\subseteq B_2^T(T^{-n}(\Pi(\jjj),\iii),C\delta\|(f_{\jjj|_n})'\|^{-1}). \label{eq:containfor2}
\end{align}
for all $\delta>0$ and $n\in\N$. The inclusions that there exists a constant $c>0$ such that
\begin{equation}\label{eq:contxi2}
\begin{split}
B((\mathbf{x},\iii),c\rho^n)\cap\xi^2(\mathbf{x},\iii)&\subseteq\mathcal{P}^n(\mathbf{x},\iii)\cap\xi^2(\mathbf{x},\iii) \\ 
&\subseteq B((\mathbf{x},\iii),c^{-1}\rho^n)\cap\xi^2(\mathbf{x},\iii)
\end{split}
\end{equation}
for all $n\in\N$ are straightforward by the strong separation condition of $\hat{\Phi}$.

Adapting the modifications explained above, one can now repeat the arguments of \cite[Lemma~7.4, Proposition~7.3, and Theorems~7.2 and 7.1]{BaranyKaenmaki2017} to obtain the claim. Nevertheless, let us go through the main steps of the proof. By using \eqref{eq:contxi2} and \eqref{eq:invmeas}, we have
\[
\begin{split}
\frac{\log\nu_{(\mathbf{x},\iii)}^{\xi^2}(B(\mathbf{x},\rho^n))}{n\log\rho} &\approx \frac{\log\nu_{(\mathbf{x},\iii)}^{\xi^2}(\mathcal{P}^n(\mathbf{x},\iii))}{n\log\rho}
=\frac{1}{n\log\rho}\sum_{k=0}^{n-1}\log\frac{\nu^{\xi^2}(\mathcal{P}^{k+1}(\mathbf{x},\iii))}{\nu^{\xi^2}(\mathcal{P}^{k}(\mathbf{x},\iii))} \\
&= \frac{1}{n\log\rho}\sum_{k=0}^{n-1}\log\nu^{\xi^2}_{T^{-k}(\mathbf{x},\iii)}(\mathcal{P}(T^{-k}(\mathbf{x},\iii))) \\ 
&\to \frac{\int\log\nu_{(\mathbf{y},\jjj)}^{\xi^2}(\mathcal{P}(\mathbf{y},\jjj))d\nu(\mathbf{y},\iii)}{\log\rho}
\end{split}
\]
as $n \to \infty$ for $\nu$-almost all $(\mathbf{x},\iii)$. Then for every $n\in\N$, by using \eqref{eq:containment}, \eqref{eq:toharmonic}, \eqref{eq:containfor2}, and Maker's ergodic theorem, we have
\begin{align*}
	&\frac{\log\nu_{(\Pi(\jjj),\iii)}^{\xi^1}(B_2^T((\Pi(\jjj),\iii),\|f_{\jjj|_{k}}'\|)}{\log\|f_{\jjj|_{k}}'\|}\\
	&\quad= \frac{1}{\log\|f_{\jjj|_{k}}'\|}\sum_{\ell=0}^{k-1}\log\frac{\nu_{T^{-\ell}(\Pi(\jjj),\iii)}^{\xi^1}(T^{-\ell}(B_2^T((\Pi(\jjj),\iii),\|f_{\jjj|_{k}}'\|)))}{\nu_{T^{-(\ell+1)}(\Pi(\jjj),\iii)}^{\xi^1}(T^{-(\ell+1)}(B_2^T((\Pi(\jjj),\iii),\|f_{\jjj|_{k}}'\|)))}\\
	&\quad\approx \frac{1}{\log\|f_{\jjj|_{k}}'\|}\sum_{\ell=0}^{k-1}\log\frac{\nu_{T^{-\ell}(\Pi(\jjj),\iii)}^{\xi^1}(B_2^T(T^{-\ell}(\Pi(\jjj),\iii),\|f_{\sigma^{\ell}\jjj|_{k-\ell}}'\|)}{\nu_{T^{-(\ell+1)}(\Pi(\jjj),\iii)}^{\xi^1}(T^{-n}(B_2^T(T^{-\ell}(\Pi(\jjj),\iii),\|f_{\sigma^{\ell}\jjj|_{k-\ell}}'\|)}\\
  &\quad= \frac{1}{\log\|f_{\jjj|_{k}}'\|}\sum_{\ell=0}^{k-1}\biggl(\log\frac{\nu_{T^{-\ell}(\Pi(\jjj),\iii)}^{\xi^1}(B_2^T(T^{-\ell}(\Pi(\jjj),\iii),\|f_{\sigma^{\ell}\jjj|_{k-\ell}}'\|)}{\nu_{T^{-\ell}(\Pi(\jjj),\iii)}^{\xi^1}(B_2^T(T^{-\ell}(\Pi(\jjj),\iii),\|f_{\sigma^{\ell}\jjj|_{k-\ell}}'\|)\cap\mathcal{P}(T^{-\ell}(\Pi(\jjj),\iii)))} \\
  &\quad\qquad\qquad\qquad\;\;\,+\log\frac{\nu_{T^{-\ell}(\Pi(\jjj),\iii)}^{\xi^1}(\mathcal{P}(T^{-\ell}(\Pi(\jjj),\iii)))}{\nu_{T^{-\ell}(\Pi(\jjj),\iii)}^{\xi^1}(B_2^T(T^{-\ell}(\Pi(\jjj),\iii),\|f_{\sigma^{\ell}\jjj|_{k-\ell}}'\|)\cap\mathcal{P}(T^{-\ell}(\Pi(\jjj),\iii)))}\biggr) \\
	&\quad\to \frac{-\int\log\nu_{(\mathbf{y},\iii)}^{\xi^1}(\mathcal{P}(\mathbf{y},\iii))\dd\nu(\mathbf{y},\iii)+\int\log\nu_{(\mathbf{y},\iii)}^{\xi^2}(\mathcal{P}(\mathbf{y},\iii))\dd\nu(\mathbf{y},\iii)}{\chi_2(\mu)}
\end{align*}
as $k \to \infty$ for $\nu$-almost all $(\Pi(\jjj),\iii)$. One can show similarly that
\begin{align*}
  &\frac{\log\nu_{(\Pi(\jjj),\iii)}^{\xi^0}(B_1^T((\Pi(\jjj),\iii),\|(g_{\jjj|_{k}})'_y\|)}{\log\|(g_{\jjj|_{k}})'_y\|} \\ 
  &\qquad\qquad\to\frac{-\int\log\nu_{(\mathbf{y},\iii)}^{\xi^0}(\mathcal{P}(\mathbf{y},\iii))d\nu(\mathbf{y},\iii)+\int\log\nu_{(\mathbf{y},\iii)}^{\xi^1}(\mathcal{P}(\mathbf{y},\iii))d\nu(\mathbf{y},\iii)}{\chi_1(\mu)}
\end{align*}
as $k \to \infty$ for $\nu$-almost every $(\Pi(\jjj),\iii)$. It is easy to see that
$$
-\int\log\nu_{(\mathbf{y},\iii)}^{\xi^0}(\mathcal{P}(\mathbf{y},\iii))\dd\nu(\mathbf{y},\iii)=h(\mu)
$$
and 
$$
-\int\log\nu_{(\mathbf{y},\iii)}^{\xi^2}(\mathcal{P}(\mathbf{y},\iii))\dd\nu(\mathbf{y},\iii)=H(\mu).
$$
Recalling that the measures $(P_2)_*\nu_{(\mathbf{x},\iii)}^{\xi^0}$ and $m$ are equivalent as also the measures $(P_2)_*\nu_{(\mathbf{x},\iii)}^{\xi^1}$ and $m_{(P_2(\mathbf{x}),\iii)}^{\proj}$, we see that
$$
\dim(m_{(x,\iii)}^{\proj})=\dimloc(m_{(x,\iii)}^{\proj},x)=\frac{-\int\log\nu_{(\mathbf{y},\iii)}^{\xi^1}(\mathcal{P}(\mathbf{y},\iii))\dd\nu(\mathbf{y},\iii)-H(\mu)}{\chi_2(\mu)}
$$
and
$$
\dim((\proj_{\iii})_*\mu)=\dimloc((\proj_{\iii})_*m,\proj_{\iii}(x))=\frac{h(\mu)+\int\log\nu_{(\mathbf{y},\iii)}^{\xi^1}(\mathcal{P}(\mathbf{y},\iii))\dd\nu(\mathbf{y},\iii)}{\chi_1(\mu)}
$$
for $m\times\overleftarrow{\mu}$-almost all $(x,\iii)$. Applying \cite[Lemma~11.3.1]{LedrappierYoung2}, we conclude that
\begin{align*}
\ldimh(m)&\geq\dim(m_{(x,\iii)}^{\proj})+\dim((\proj_{\iii})_*\mu)\\
&=\frac{h(\mu)-H(\mu)}{\chi_2(\mu)}+\biggl(1-\frac{\chi_1(\mu)}{\chi_2(\mu)}\biggr)\frac{h(\mu)+\int\log\nu_{(\mathbf{y},\iii)}^{\xi^1}(\mathcal{P}(\mathbf{y},\iii))\dd\nu(\mathbf{y},\iii)}{\chi_1(\mu)}\\
&=\frac{h(\mu)-H(\mu)}{\chi_2(\mu)}+\biggl(1-\frac{\chi_1(\mu)}{\chi_2(\mu)}\biggr)\dim((\proj_\iii)_*\mu)
\end{align*}
for $m\times\overleftarrow{\mu}$-almost all $(x,\iii)$. This is what we wanted to prove.
\end{proof}

\section{Main theorem}\label{sec:main}

In this section, we introduce a general parametrized iterated function system satisfying \ref{it:condLY1} and \ref{it:condLY2} and state the main theorem for it. Fix $d,m \in \N$ and let $V$ and $W$ be open and bounded subsets of $\R^d$ and $\R^m$, respectively. For each $\vvv\in V$, $\www\in W$, and $i\in\{1,\ldots,N\}$, let $f_i^{\vvv}\colon[0,1]\to[0,1]$ and $g_i^{\www}\colon[0,1]^2\to[0,1]$ be such that
\begin{enumerate}[label=(T\arabic*)]
	\item\label{eq:cond1} the maps $x\mapsto f_i^{\vvv}(x)$ and $(x,y)\mapsto g_i^{\www}(x,y)$ are two times continuously differentiable and there exists a constant $C>0$ such that
	$$
	\max\{|(g_i^{\www})'_x(x,y)|,|(g_i^{\www})''_{xy}(x,y)|,|(g_i^{\www})''_{yy}(x,y)|\}<C
	$$
	for all $(x,y)\in[0,1]^2$ and $\www\in W$,
	\item\label{eq:cond2} there exist $\rho,\gamma,\tau\in(0,1)$ such that
	$$
    \tau < \gamma^{-1}|(f_i^\vvv)'(x)| < |(g_i^{\www})'_y(x,y)| < \rho
	$$
  for all $(x,y)\in[0,1]^2$ and $(\vvv,\www)\in V\times W$,
	\item\label{eq:cond3} for each $(x,y)\in[0,1]^2$, the maps $\vvv\mapsto f_i^{\vvv}(x)$ and $\www\mapsto g_i^{\www}(x,y)$ are continuous on $V$ and $W$, respectively.
	\setcounter{nameOfYourChoice}{\value{enumi}}
\end{enumerate}

Define a parametrized family of planar iterated function systems by setting $\Phi^\ttt = (F_i^{\vvv,\www})_{i \in \{1,\ldots,N\}}$, where $F_i^{\vvv,\www} \colon \R^2 \to \R^2$, $F_i^{\vvv,\www}(x,y)=(f_i^{\vvv}(x),g_i^{\www}(x,y))$ and $\ttt=(\vvv,\www)\in V\times W$. Note that the conditions \ref{eq:cond1}--\ref{eq:cond3} imply \ref{it:condLY1} and \ref{it:condLY2}. Furthermore, observe that all the iterates of the first coordinates of $F_{\iii}^{\vvv,\www}$ depend only on $\vvv$ while the second coordinates of $F_{\iii}^{\vvv,\www}$ when $|\iii|\geq2$ depend on both $\vvv$ and $\www$. Hence, we denote the iterates of the coordinate functions by $f_\iii^\vvv$ and $g_\iii^{\vvv,\www}$, respectively. For such parametrized families the canonical projection defined in \eqref{eq:natproj} is denoted by $\pi_\vvv = (\pi^1_\vvv,\pi^2_\vvv)$ and the non-linear projection defined in \eqref{eq:nonlinproj} is denoted by $\proj_\iii^{\vvv,\www}$.

We assume that $\Phi^\ttt$ satisfies also the following assertions \ref{eq:cond4}--\ref{eq:cond7}.
\begin{enumerate}[label=(T\arabic*)]
	\setcounter{enumi}{\value{nameOfYourChoice}}
	\item\label{eq:cond4} for every $\ttt_0=(\vvv_0,\www_0)\in V\times W$ and $\varepsilon>0$ there exists $\delta>0$ such that
	$$
	e^{-\varepsilon|\iii|}\leq\frac{|(g_{\iii}^{\ttt})'_y(x,y)|}{|(g_{\iii}^{\ttt_0})'_y(x,y)|}\leq e^{\varepsilon|\iii|}
	$$
  for all $\iii\in\Sigma_*$, $\ttt\in B(\ttt_0,\delta)$, and $(x,y)\in[0,1]^2$,
	\item\label{eq:cond5} for every $\vvv_0\in V$ there exists $C(\vvv_0)>0$ such that the following transversality condition holds
	$$
	\mathcal{L}^m(\{\www\in W:|\proj_\kkk^{\vvv_0,\www}(\iii)-\proj_\kkk^{\vvv_0,\www}(\jjj)|<r\}) \leq C(\vvv_0)r
	$$
  for all $\iii,\jjj,\kkk\in\Sigma$ with $i_1\neq j_1$,
	\item\label{eq:cond6} for every $\varepsilon>0$ and $\vvv_0\in V$ there exists $\delta>0$ such that
	$$
	e^{-\varepsilon|\iii|}\leq\frac{|(f_{\iii}^\vvv)'(x)|}{|(f_{\iii}^{\vvv_0})'(x)|}\leq e^{\varepsilon|\iii|}
	$$
  for all $\iii\in\Sigma_*$, $\vvv\in B\big(\vvv_0,\delta\big)$, and $x\in[0,1]$,
	\item\label{eq:cond7} there exists $C>0$ such that
	$$
	\mathcal{L}^d(\{\vvv\in V:|\pi^1_\vvv(\iii)-\pi^1_\vvv(\jjj)|<r\})\leq Cr
	$$
  for all $\iii,\jjj\in\Sigma$ with $i_1\neq j_1$.
\end{enumerate}
Let $s_0(\ttt)$ be the unique root of the subadditive pressure function $P_{\ttt}$ defined in \eqref{eq:subpressure} and recall that the Lyapunov dimension defined in \eqref{eq:upperboundformeasure} is
\begin{equation*} 
  \diml(\mu,\ttt) = \min\biggl\{2,\frac{h(\mu)}{\chi_1(\mu,\ttt)},1+\frac{h(\mu)-\chi_1(\mu,\ttt)}{\chi_2(\mu,\vvv)}\biggr\},
\end{equation*}
where we have emphasized the dependence on the parameter $\ttt=(\vvv,\www)\in V\times W$.

\begin{theorem}\label{thm:meta}
  Suppose that $\Phi^{\ttt}$ satisfies the assumptions \ref{eq:cond1}--\ref{eq:cond7}. If $\mu \in \MM_\sigma(\Sigma)$ is quasi-Bernoulli, then
  \begin{equation*}
    \dim((\pi_\ttt)_*\mu) = \diml(\mu,\ttt)
  \end{equation*}
  for $\LL^{d+m}$-almost all $\ttt\in V \times W$ and $(\pi_\ttt)_*\mu\ll\LL^2$ for $\LL^{d+m}$-almost all $\ttt \in V \times W$ for which $h(\mu)>\chi_1(\mu,\ttt)+\chi_2(\mu,\vvv)$. Furthermore,
  \begin{equation*}
    \dimh(X_\ttt) = \udimm(X_\ttt) = \min\{2,s_0(\ttt)\}
  \end{equation*}
  for $\LL^{d+m}$-almost all $\ttt\in V \times W$ and $\LL^2(X_\ttt)>0$ for $\LL^{d+m}$-almost all $\ttt \in V \times W$ with $s_0(\ttt)>2$.
\end{theorem}

\subsection{Components of the proof}

The proof of the main theorem, Theorem~\ref{thm:meta}, can be decomposed to the proof of the following three propositions which together with the weak Ledrappier-Young formula, Theorem~\ref{thm:LY}, then imply the claim.

\begin{proposition}\label{prop:main1}
  Suppose that $\Phi^{\ttt}$ satisfies the assumptions \ref{eq:cond1}--\ref{eq:cond6}. If $\mu \in \MM_\sigma(\Sigma)$, then
  $$
    \ldimh((\proj_\iii^{(\vvv_0,\www)})_*\mu) \ge \min\biggl\{1,\frac{h(\mu)}{\chi_1(\mu,{(\vvv_0,\www)})}\biggr\}
  $$
  for all $\vvv_0\in V$, for $\LL^m$-almost all $\www\in W$ and for $\overleftarrow{\mu}$-almost all $\iii \in \Sigma$.
\end{proposition}

\begin{proposition}\label{prop:main2}
  Suppose that $\Phi^{\ttt}$ satisfies the assumptions \ref{eq:cond1}--\ref{eq:cond5}. If $\mu \in \MM_\sigma(\Sigma)$ and there exists $\vvv_0\in V$ such that
  \begin{equation*} 
	  \ldimh((\pi^1_{\vvv_0})_*\mu) > \frac{h(\mu)-\chi_1(\mu,(\vvv_0,\www))}{\chi_2(\mu,\vvv_0)}
  \end{equation*}
  for all $\www\in W$, then
	$$
	  H_{\vvv_0,\www}(\mu)=0
	$$
  for $\LL^m$-almost all $\www\in W$, where $H_\ttt(\mu)=-\int\log\mu_\jjj^{\pi_\ttt}([j_1])\dd\mu(\jjj)$.
\end{proposition}

\begin{proposition}\label{prop:main3}
  Suppose that $\Phi^{\ttt}$ satisfies the assumptions \ref{eq:cond1}--\ref{eq:cond5}. If $\mu \in \MM_\sigma(\Sigma)$ and there exists $\vvv_0\in V$ such that
  \begin{equation*} 
	  (\pi^1_{\vvv_0})_*\mu\ll\LL^1,
  \end{equation*}
  then
	$$
	  (\pi_{(\vvv_0,\www)}^2)_*\mu^{\pi^1_{\vvv_0}}_{\iii}\ll\LL^1
	$$
  for $\LL^m$-almost all $\www\in W$ and for $\mu$-almost all $\iii \in \Sigma$, where $\{\mu_{\iii}^{\pi_{\vvv_0}^1}\}_{\iii\in\Sigma}$ is the family of conditional measures supported on the partition element $(\pi_{\vvv_0}^1)^{-1}(\pi_{\vvv_0}^1(\iii))$.
\end{proposition}

Let us next demonstrate how Theorem~\ref{thm:meta} follows from Propositions~\ref{prop:main1}--\ref{prop:main3} and the weak Ledrappier-Young formula, Theorem~\ref{thm:LY}. To that end, we require the following lemma.

\begin{lemma}\label{lem:cont}
  Suppose that $\Phi^{\ttt}$ satisfies the assumptions \ref{eq:cond1}--\ref{eq:cond4} and \ref{eq:cond6}. If $\mu \in \MM_\sigma(\Sigma)$, then the maps $\ttt\mapsto s_0(\ttt)$, $\ttt\mapsto\chi_1(\mu,\ttt)$, and $\vvv\mapsto\chi_2(\mu,\vvv)$ are continuous. In particular, the map $\ttt\mapsto\diml(\mu,\ttt)$ is continuous.
\end{lemma}

\begin{proof}
	Although the proof is straightforward, we present it for completeness. Fix $\ttt_0=(\vvv_0,\www_0)\in V\times W$. Then, by \ref{eq:cond4} and \ref{eq:cond6}, for every $\varepsilon>0$ there exists $\delta>0$ such that
	$$
	e^{-sn\varepsilon}\leq\frac{\sum_{\iii\in\Sigma_n}\varphi^s_{\ttt}(\iii,x,y)}{\sum_{\iii\in\Sigma_n}\varphi^s_{\ttt_0}(\iii,x,y)}\leq e^{sn\varepsilon}
	$$
  for all $\ttt\in B(\ttt_0,\delta)$ and $s\in[0,\infty)$, where $\varphi^s_{\ttt}(\iii,x,y)$ is defined in \eqref{eq:singvaluefunction}. It follows that $|P_{\ttt}(s)-P_{\ttt_0}(s)| < s\varepsilon$. On the other hand, \ref{eq:cond2} implies that $|P_{\ttt}(s)-P_{\ttt}(s')| \ge L|s-s'|$ for every $\ttt\in V\times W$, where $L=-\log\rho$. Therefore,
	\begin{align*}
		0&=|P_{\ttt}(s_0(\ttt))-P_{\ttt_0}(s_0(\ttt_0))|\\
		&\geq|P_{\ttt}(s_0(\ttt))-P_{\ttt}(s_0(\ttt_0))|-|P_{\ttt}(s_0(\ttt_0))-P_{\ttt_0}(s_0(\ttt_0))|\\
		&\geq L|s_0(\ttt)-s(\ttt_0)|-s_0(\ttt_0)\varepsilon,
	\end{align*}
	which proves the continuity of $\ttt\mapsto s_0(\ttt)$.
	
	The continuity of the Lyapunov exponents follows by the continuity of the map $\ttt\mapsto\pi_{\ttt}(\iii)$ and the dominated convergence theorem, which can be applied by \ref{eq:cond2}. The continuity of $\ttt\mapsto\pi_{\ttt}(\iii)$ follows from \ref{eq:cond1}--\ref{eq:cond3}. Indeed, notice first that $\vvv\mapsto\pi_{\vvv}^1(\iii)$ is continuous by \cite[proof of Lemma~4.1]{SSU01} and
	\begin{align*}
		|\pi_{\ttt}^2(\iii)-\pi_{\ttt_0}^2(\iii)|&=|g_{i_1}^{\www}(\pi_{\ttt}(\sigma\iii))-g_{i_1}^{\www_0}(\pi_{\ttt_0}(\sigma\iii))|\\
		&=|g_{i_1}^{\www}(\pi_{\ttt}(\sigma\iii))-g_{i_1}^{\www}(\pi_{\ttt_0}(\sigma\iii))+g_{i_1}^{\www}(\pi_{\ttt_0}(\sigma\iii))-g_{i_1}^{\www_0}(\pi_{\ttt_0}(\sigma\iii))|\\
		&\leq C|\pi_{\vvv}^1(\sigma\iii)-\pi_{\vvv_0}^1(\sigma\iii)|+|g_{i_1}^{\www}(\pi_{\ttt_0}(\sigma\iii))-g_{i_1}^{\www_0}(\pi_{\ttt_0}(\sigma\iii))|\\
		&\qquad\qquad+\rho|\pi_{\ttt}^2(\sigma\iii)-\pi_{\ttt_0}^2(\sigma\iii)|.
	\end{align*}
	Hence, by the induction,
	$$
	|\pi_{\ttt}^2(\iii)-\pi_{\ttt_0}^2(\iii)|\leq\sum_{k=1}^\infty(C|\pi_{\vvv}^1(\sigma^k\iii)-\pi_{\vvv_0}^1(\sigma^k\iii)|+|g_{i_k}^{\www}(\pi_{\ttt_0}(\sigma^k\iii))-g_{i_k}^{\www_0}(\pi_{\ttt_0}(\sigma^k\iii))|)\rho^{k-1},
	$$
	which can be rendered arbitrary small by the continuity of $\vvv\mapsto\pi_{\vvv}^1(\iii)$ and \ref{eq:cond2}.
\end{proof}

Observe that for every $\ttt_0\in V\times W$, there exists a unique quasi-Bernoulli $\mu_{\ttt_0} \in \MM_\sigma(\Sigma)$ such that
\begin{equation}\label{eq:equimeasure}
  s_0(\ttt_0)=\diml(\mu_{\ttt_0},\ttt_0).
\end{equation}
This follows by considering a H\"older-continuous potential $\iii\mapsto\log\varphi^{s(\ttt_0)}(i_1,\pi_{\ttt_0}(\sigma\iii))$ and invoking \cite[Theorem~1.22]{Bowen}. We are now ready to prove the main theorem.

\begin{proof}[Proof of Theorem~\ref{thm:meta}]
	Let us first prove the claims for the quasi-Bernoulli measure. By \cite[Theorem 7.2]{SSU01}, assuming \ref{eq:cond1}--\ref{eq:cond4} and \ref{eq:cond7}, there exists $\tilde{V}\subseteq V$ such that $\LL^d(V\setminus\tilde{V})=0$ and, in particular, $\ldimh((\pi^1_{\vvv})_*\mu)=\min\{1,\frac{h(\mu)}{\chi_2(\mu,\vvv)}\}$ for all $\vvv\in\tilde{V}$ and $(\pi^1_{\vvv})_*\mu\ll\mathcal{L}$ for all $\vvv\in \tilde{V}$ for which $h(\mu)>\chi_2(\mu,\vvv)$. By \eqref{eq:upperboundformeasure} and Fubini's Theorem, it is enough to show that, for each $\vvv_0\in\tilde{V}$, we have
	\begin{equation}\label{eq:en1}
	  \ldimh((\pi_{\vvv_0,\www})_*\mu)\ge\diml(\mu,(\vvv_0,\www))
	\end{equation}
  for $\LL^m$-almost every $\www\in W$. Fix $\vvv_0\in\tilde{V}$. It is easy to see that on the region of $W$, where $\diml(\mu,(\vvv_0,\www))\leq1$, we have
	$$
	\diml(\mu,(\vvv_0,\www))=\min\biggl\{1,\frac{h(\mu)}{\chi_1(\mu,(\vvv_0,\www))}\biggr\}.
	$$
	Hence, the claim \eqref{eq:en1} follows from Proposition~\ref{prop:main1} and the fact that $\ldimh(\proj_\iii^{\vvv_0,\www})_*\mu \leq \ldimh(\pi_{\vvv_0,\www})_*\mu$. On the region of $W$, where $1\leq\diml(\mu,(\vvv_0,\www))<2$, simple algebraic manipulations show that $\min\{1,\frac{h(\mu)}{\chi_2(\mu,\vvv_0)}\}>\frac{h(\mu)-\chi_1(\mu,(\vvv_0,\www))}{\chi_2(\mu,\vvv_0)}$ and hence, \eqref{eq:en1} follows by Theorem~\ref{thm:LY}, Proposition~\ref{prop:main1}, and Proposition~\ref{prop:main2}. Finally, on the region of $W$, where $\diml(\mu,(\vvv_0,\www))>2$, we have
  $$
    h(\mu)>\chi_1(\mu,(\vvv_0,\www))+\chi_2(\mu,\vvv_0)>\chi_2(\mu,\vvv_0)
  $$
  and so, by Proposition~\ref{prop:main3}, we see that $(\pi_{\vvv_0,\www}^2)_*\mu^{\pi^1_{\vvv_0,\www}}_{\iii}\ll\LL^1$ for $\LL^m$-almost all $\www\in W$. Since
	\begin{align*}
    (\pi_{\vvv_0,\www})_*\mu&=\int(\pi_{\vvv_0,\www})_*\mu^{\pi^1_{\vvv_0}}_{\iii}\dd\mu(\iii)\\
    &=\int\delta_{\pi^1_{\vvv_0}(\iii)}\times(\pi_{\vvv_0,\www}^2)_*\mu^{\pi^1_{\vvv_0}}_{\iii}\dd\mu(\iii)=\int\delta_{\iii}\times(\pi_{\vvv_0,\www}^2)_*\mu^{\pi^1_{\vvv_0}}_{\iii}\dd(\pi^1_{\vvv_0})_*\mu(\iii),
	\end{align*}
	we get, by Fubini's Theorem, that $(\pi_{\vvv_0,\www})_*\mu\ll\LL^2$ for $\LL^m$-almost all $\www\in W$. Applying Fubini's Theorem once more, we conclude that $(\pi_\ttt)_*\mu\ll\LL^2$ for $\LL^{d+m}$-almost all $\ttt\in V\times W$.
	
  Let us then prove the claims for the attractor. They basically follow by the continuity properties of Lemma~\ref{lem:cont} for the measure $\mu_{\ttt_0}$ defined in \eqref{eq:equimeasure}. Let us first assume that $s_0(\ttt) \le 2$ for all $\ttt \in V \times W$. Recalling \eqref{eq:upperboundforset}, let us argue by contradiction and suppose that $\LL^{d+m}(\{\ttt\in V\times W:\dimh(X_{\ttt})<s_0(\ttt)\})>0$. Then there exists $n \in \N$ such that $\LL^{d+m}(\{\ttt\in V\times W:\dimh(X_{\ttt})<s_0(\ttt)-1/n\})>0$. Let $\ttt_0$ be a Lebesgue density point of that set. Hence, for $\LL^{d+m}$-almost every $\ttt\in\{\ttt\in V\times W:\dimh(X_{\ttt})<s_0(\ttt)-1/n\}$, we have
	$$
	s_0(\ttt)-1/n>\dimh(X_\ttt)\geq\dim((\pi_\ttt)_*\mu_{\ttt_0})=\diml(\mu_{\ttt_0},\ttt).
	$$
	But by Lemma~\ref{lem:cont}, $\diml(\mu_{\ttt_0},\ttt)\to s(\ttt_0)$ as $\ttt\to\ttt_0$, which contradicts to the continuity of $\ttt\mapsto s_0(\ttt)$ at $\ttt_0$. Recalling \eqref{eq:upperboundforset}, this shows that $\dimh(X_\ttt) = \udimm(X_\ttt) = s_0(\ttt)$ for $\LL^{d+m}$-almost all $\ttt \in V \times W$ provided that $s_0(\ttt) \le 2$ for all $\ttt \in V \times W$. If $s(\ttt_0)>2$ then again by Lemma~\ref{lem:cont}, we have $\diml(\mu_{\ttt_0},\ttt)>2$ in a sufficiently small neighbourhood of $\ttt_0$ and therefore, $(\pi_{\ttt})_*\mu_{\ttt_0}\ll\LL^2$ and $\LL^2(X_\ttt)>0$ for $\LL^{d+m}$-almost every $\ttt$ in this neighborhood completing the proof.
\end{proof}

\section{Proofs of the propositions} \label{sec:proof-props}

In this section, we prove Propositions \ref{prop:main1}--\ref{prop:main3} and hence, finish the proof of Theorem \ref{thm:meta}. Let $\Phi^\ttt$ be a parametrized planar iterated function system as described in Section \ref{sec:main}. Before going into the proofs, we study the non-linear projection \eqref{eq:nonlinproj} in more detail. We denote the strong-stable tangent bundle \eqref{eq:defofbundle} by $u_\ttt$ and the strong-stable foliation \eqref{eq:diffeq} by $y_\ttt$. It is easy to see that

\begin{equation}\label{eq:uposy}
\begin{split}
(u_\ttt)'_y(\iii,x,y)&=\sum_{k=1}^\infty\Biggl(\frac{-(g_{i_k}^\www)_{xy}''(F_{\overleftarrow{\iii|_{k-1}}}^\ttt(x,y))(f_{\overleftarrow{\iii|_{k-1}}}^\vvv)'(x)}{(g_{\overleftarrow{\iii|_{k}}}^\ttt)_y'(x,y)}(g_{\overleftarrow{\iii|_{k-1}}}^\ttt)_y'(x,y) \\
&\qquad\qquad+\sum_{\ell=1}^k\frac{(g_{i_k}^\www)_x'(F_{\overleftarrow{\iii|_{k-1}}}^\ttt(x,y))(f_{\overleftarrow{\iii|_{k-1}}}^\vvv)'(x)}{(g_{\overleftarrow{\iii|_{k}}}^\ttt)_y'(x,y)}\\ 
&\qquad\qquad\qquad\qquad\cdot\frac{(g_{i_\ell}^\www)_{yy}''(F_{\overleftarrow{\iii|_{\ell-1}}}^\ttt(x,y))}{(g_{i_\ell}^\www)_{y}'(F_{\overleftarrow{\iii|_{\ell-1}}}^\ttt(x,y))}(g_{\overleftarrow{\iii|_{\ell-1}}}^\ttt)_y'(x,y)\Biggr).
\end{split}
\end{equation}
By \ref{eq:cond1} and \ref{eq:cond2}, there exists a constant $C>0$ such that
\begin{equation}\label{eq:ubounded}
|(u_\ttt)'_y(\iii,x,y)|\leq C
\end{equation}
for all $(x,y)\in[0,1]^2$, $\iii\in\Sigma$, and $\ttt\in V\times W$.

\begin{lemma}\label{lem:dist}
	There exists $C>0$ such that
	$$
	|y_\ttt(\iii,(x_0,y_0),x)-y_\ttt(\iii,(x_1,y_1),x)|\leq C|y_\ttt(\iii,(x_0,y_0),0)-y_\ttt(\iii,(x_1,y_1),0)|
	$$
  for all $\iii\in\Sigma$, $(x_0,y_0),(x_1,y_1)\in[0,1]^2$, $x\in[0,1]$, and $\ttt\in V\times W$.
\end{lemma}

\begin{proof}
	By the uniqueness of the solution of the differential equation \eqref{eq:diffeq}, either
  \begin{equation*}
    y_\ttt(\iii,(x_0,y_0),x)\equiv y_\ttt(\iii,(x_1,y_1),x)
  \end{equation*}
  or
  \begin{equation*}
    y_\ttt(\iii,(x_0,y_0),x)\neq y_\ttt(\iii,(x_1,y_1),x)
  \end{equation*}
  for all $x\in[0,1]$. In the first case, the claim of the lemma is trivial. In the second case, without loss of generality, we may assume that
	$$
    y_\ttt(\iii,(x_0,y_0),x)> y_\ttt(\iii,(x_1,y_1),x)
	$$
  for all $x\in[0,1]$. By Lagrange's Mean Value Theorem, for each $x\in[0,1]$ there exists $\xi$ (which might also depend on $\iii$, $\ttt$, $(x_0,y_0)$, and $(x_1,y_1)$) such that
	\begin{align*}
		y_\ttt'(\iii,(x_0,y_0),x)-y_\ttt'(\iii,(x_1,y_1),x)&=u_\ttt(\iii,x,y_\ttt(\iii,(x_0,y_0),x))-u_\ttt(\iii,x,y_\ttt(\iii,(x_1,y_1),x))\\
		&=(u_\ttt)'_y(\iii,x,\xi)\cdot(y_\ttt(\iii,(x_0,y_0),x)-y_\ttt(\iii,(x_1,y_1),x))\\
		&\leq C(y_\ttt(\iii,(x_0,y_0),x)-y_\ttt(\iii,(x_1,y_1),x)),
	\end{align*}
	where in the last inequality we applied \eqref{eq:ubounded}. Hence, by Gr\"onwall's inequality (see, for example, \cite[Theorem~1.2.1]{Diffineq}), we have
	$$
    y_\ttt(\iii,(x_0,y_0),x)-y_\ttt(\iii,(x_1,y_1),x)\leq(y_\ttt(\iii,(x_0,y_0),0)-y_\ttt(\iii,(x_1,y_1),0))e^{C x}.
	$$
  The claim follows as $x\in[0,1]$.
\end{proof}

Since the first parameter-coordinate $\vvv_0\in V$ will be fixed throughout the proofs, with a slight abuse of notation, we denote $g_{\iii}^{\vvv_0,\www}$ by $g_{\iii}^{\www}$ and $f_\iii^{\vvv_0}$ by $f_\iii$ for all $\iii\in\Sigma_*$. Similarly, the canonical projection $\pi_{\vvv_0,\www}$ is denoted by $\pi_{\www} = (\pi^1,\pi^2_\www)$ and the non-linear projection $\proj_{\iii}^{\vvv_0,\www}$ by $\proj_{\iii}^{\www}$ for all $\iii\in\Sigma$.

\begin{proof}[Proof of Proposition~\ref{prop:main1}]
	Fix $\vvv_0\in V$. Standard calculations show that for every $\varepsilon>0$ there exists $\delta>0$ such that, for every $\www\in B(\www_0,\delta)$, $\iii\in\Sigma$, and $n\geq1$, we have
	\begin{equation}\label{eq:boundedparam}
	e^{-n\varepsilon}\leq\frac{(g_{\iii|_n}^{\www_0})'_y(\pi_{\www_0}(\sigma^n\iii))}{(g_{\iii|_n}^{\www})'_y(\pi_{\www}(\sigma^n\iii))}\leq e^{n\varepsilon},
	\end{equation}
	by the condition \ref{eq:cond4}. We will show that for every $\varepsilon>0$ there exists a $\delta>0$ such that, for any $\www_0\in W$ and $\mathcal{L}^m\times\overleftarrow{\mu}$-almost every $(\www,\iii)\in B(\www_0,\delta)\times\Sigma$, we have
	\begin{equation}\label{eq:en2}
	\ldimh(\proj_\iii^{\www})_*\mu\geq\min\biggl\{1-\varepsilon,\frac{h(\mu)-\varepsilon}{\chi_1(\mu,(\vvv_0,\www_0))+2\varepsilon}\biggr\}.
	\end{equation}	
	From this, using the continuity of the Lyapunov exponent given by Lemma~\ref{lem:cont}, one can show, by a standard density argument, that 
	$$
	\ldimh(\proj_\iii^{\www})_*\mu\ge\min\biggl\{1,\frac{h(\mu)}{\chi_1(\mu,(\vvv_0,\www))}\biggr\}
	$$
  for $\mathcal{L}^m\times\overleftarrow{\mu}$-almost all $(\www,\iii)$. For details, see Simon, Solomyak, and Urba\'nski \cite[proof of Theorem~2.3(i)]{SSU01}. 
	
	We define $E=\bigcap_{\varepsilon>0}\bigcup_{M=1}^\infty E_{M,\varepsilon}$, where
	\begin{equation}\label{eq:egorov}
	\begin{split}
		E_{M,\varepsilon}=\{\iii\in\Sigma:\;&e^{-n(\chi_1(\mu,\ttt_0)+\varepsilon)}\leq(g_{\iii|_n}^{\www_0})'_y(\pi_{\www_0}(\sigma^n\iii))\leq e^{-n(\chi_1(\mu,\ttt_0)-\varepsilon)}\text{ and }\\
		&e^{-n(h(\mu)+\varepsilon)}\leq\mu([\iii|_n])\leq e^{-n(h(\mu)-\varepsilon)}\text{ for all }n\geq M\}.
	\end{split}
	\end{equation}
	Notice that, by Birkhoff's Ergodic Theorem and Shannon-McMillan-Breiman Theorem, we have $\mu(E)=1$. Fix $\www_0\in W$ and let
  \begin{equation*}
    s<\min\biggl\{1-\varepsilon,\frac{h(\mu)-\varepsilon}{\chi_1(\mu,\vvv_0,\www_0)+2\varepsilon}\biggr\}.
  \end{equation*}
  Choose $M\geq1$ such that $\mu(E_{M,\varepsilon})>1-\varepsilon$. Let $A_n=\{(\iii,\jjj)\in \Sigma\times E_{\varepsilon,M}:|\iii\wedge\jjj|=n\}$ and observe that
	\begin{align*}
		\int_{B(\www_0,\delta)}&\iiint_{\Sigma\times E_{\varepsilon,M}}\frac{\dd\mu(\jjj)\dd\mu(\hhh)\dd\overleftarrow{\mu}(\iii)\dd\mathcal{L}^m(\www)}{|\proj_{\iii}^\www(\jjj)-\proj_{\iii}^\www(\hhh)|^s}\\
		&=\sum_{n=0}^\infty\int_{B(\www_0,\delta)}\iiint_{A_n}\frac{\dd\mu(\jjj)d\mu(\hhh)\dd\overleftarrow{\mu}(\iii)\dd\mathcal{L}^m(\www)}{|\proj_{\iii}^\www(\jjj)-\proj_{\iii}^\www(\hhh)|^s}\\
		&=\sum_{n=0}^\infty\iiint_{A_n}\int_{B(\www_0,\delta)}\frac{\dd\mathcal{L}^m(\www)\dd\mu(\jjj)\dd\mu(\hhh)\dd\overleftarrow{\mu}(\iii)}{|\proj_{\iii}^\www(\jjj)-\proj_{\iii}^\www(\hhh)|^s}\\
		&\leq\sum_{n=0}^\infty\iiint_{A_n}\int_0^\infty\mathcal{L}^m(\{\www\in B(\www_0,\delta):|\proj_{\iii}^\www(\jjj) \\ 
    &\qquad\qquad\qquad-\proj_{\iii}^\www(\hhh)|<r^{-1/s}\})\dd r\dd\mu(\jjj)\dd\mu(\hhh)\dd\overleftarrow{\mu}(\iii).
	\end{align*}
By Lemma~\ref{lem:invariance} and Lemma~\ref{lem:dist}, we have
\begin{align*}
	|\proj_{\iii}^\www(\jjj)-\proj_{\iii}^\www(\hhh)|&\geq C^{-1}|y_\www(\iii,\pi_\www(\jjj),f_{\jjj\wedge\hhh}(0))-y_\www(\iii,\pi_\www(\hhh),f_{\jjj\wedge\hhh}(0))|\\
	&=C^{-1}|g_{\jjj\wedge\hhh}^{\www}(0,\proj_{\overleftarrow{\jjj\wedge\hhh}\iii}^{\www}(\sigma^n\jjj))-g_{\jjj\wedge\hhh}^{\www}(0,\proj_{\overleftarrow{\jjj\wedge\hhh}\iii}^{\www}(\sigma^n\hhh))|\\
	&\geq C^{-1}\inf_{\xi\in[0,1]}|(g_{\jjj\wedge\hhh}^{\www})'_y(0,\xi)||\proj_{\overleftarrow{\jjj\wedge\hhh}\iii}^{\www}(\sigma^n\jjj)-\proj_{\overleftarrow{\jjj\wedge\hhh}\iii}^{\www}(\sigma^n\hhh)|
\intertext{and, by \eqref{eq:bd} and \eqref{eq:boundedparam},}
	&\geq C'^{-1}\inf_{\xi\in[0,1]}|(g_{\jjj\wedge\hhh}^{\www_0})'_y(0,0)|e^{-n\varepsilon}|\proj_{\overleftarrow{\jjj\wedge\hhh}\iii}^{\www}(\sigma^n\jjj)-\proj_{\overleftarrow{\jjj\wedge\hhh}\iii}^{\www}(\sigma^n\hhh)|\\
	&\geq C'^{-1}e^{-(\chi_1(\mu,\ttt_0)+2\varepsilon)n}|\proj_{\overleftarrow{\jjj\wedge\hhh}\iii}^{\www}(\sigma^n\jjj)-\proj_{\overleftarrow{\jjj\wedge\hhh}\iii}^{\www}(\sigma^n\hhh)|
\end{align*}
for all $n\geq M$, $(\jjj,\hhh)\in A_n$, and $\iii\in\Sigma$, where in the last inequality we applied the definition of the set $E_{M,\varepsilon}$. Thus, by the transversality condition \ref{eq:cond5},
\begin{align*}
&\mathcal{L}^m(\{\www\in B(\www_0,\delta):|\proj_{\iii}^\www(\jjj)-\proj_{\iii}^\www(\hhh)|<r^{-1/s}\})\\
&\qquad\leq\mathcal{L}^m(\{\www\in B(\www_0,\delta):|\proj_{\overleftarrow{\jjj\wedge\hhh}\iii}^{\www}(\sigma^n\jjj)-\proj_{\overleftarrow{\jjj\wedge\hhh}\iii}^{\www}(\sigma^n\hhh)|<e^{n(\chi_1(\mu,\ttt_0)+2\varepsilon)}C'r^{-1/s}\})\\
&\qquad\leq C''\min\{\mathcal{L}^m(B(\www_0,\delta)),e^{n(\chi_1(\mu,\ttt_0)+2\varepsilon)}r^{-1/s}\}
\end{align*}
for all $(\jjj,\hhh)\in A_n$, $\iii\in\Sigma$, and $r>0$. Therefore, we have
\begin{align*}
	\int_{B(\www_0,\delta)}&\iiint_{\Sigma\times E_{\varepsilon,M}}\frac{\dd\mu(\jjj)\dd\mu(\hhh)\dd\overleftarrow{\mu}(\iii)\dd\mathcal{L}^m(\www)}{|\proj_{\iii}^\www(\jjj)-\proj_{\iii}^\www(\hhh)|^s}\\
		&\leq\sum_{n=0}^\infty\iiint_{A_n}\int_0^\infty C''\min\{\mathcal{L}^m(B(\www_0,\delta)),\\ 
    &\qquad\qquad\qquad e^{n(\chi_1(\mu,\ttt_0)+2\varepsilon)}r^{-1/s}\}\dd r\dd\mu(\jjj)\dd\mu(\hhh)\dd\overleftarrow{\mu}(\iii)\\
		&\leq\sum_{n=0}^\infty\iint_{A_n}\tilde{C}e^{ns(\chi_1(\mu,\ttt_0)+2\varepsilon)}\dd\mu(\jjj)\dd\mu(\hhh)\\
		&\leq\sum_{n=0}^\infty\sum_{\iii\in\Sigma_n}\tilde{C}e^{n(s(\chi_1(\mu,\ttt_0)+2\varepsilon)-h(\mu)+\varepsilon)}\mu([\iii])<\infty\\
		&\leq\sum_{n=0}^\infty\tilde{C}e^{n(s(\chi_1(\mu,\ttt_0)+2\varepsilon)-h(\mu)+\varepsilon)}<\infty,
	\end{align*}
which implies \eqref{eq:en2}.
\end{proof}

To prove Proposition~\ref{prop:main2}, we utilize the method of Bara\'nski, Gutman, and \'Spiewak \cite[proof of Theorem~1.12(i)]{baranski2023regularity}.

\begin{proof}[Proof of Proposition~\ref{prop:main2}]
	Let $\vvv_0\in V$ be as in the assumption. Write $\alpha=h(\mu)-\chi_2(\mu)\ldimh(\pi^1)_*\mu$ and fix $\www_0\in W$. Choose $\varepsilon>0$ such that $\chi_1(\mu,\ttt_0)>\alpha+3\varepsilon$, where $\ttt_0=(\vvv_0,\www_0)$. It is enough to show that for $\mathcal{L}^m$-almost every $\www$ there exists $\Omega_\www\subset\Sigma$ such that $\pi_\www|_{\Omega_\www}$ is injective, and so $\mu_\iii^{\pi_\www}=\delta_\iii$ for $\mu$-almost every $\iii$. In order to do so, we will show that
	$$
	\mathcal{L}^m\times\mu(A)=0,
	$$
	where
	$$
	A=\{(\www,\iii)\in W\times\Sigma:\text{there is }\jjj\in\Sigma\setminus\{\iii\}\text{ such that }\pi_\www(\iii)=\pi_\www(\jjj)\}.
	$$
	For each $\varepsilon>0$ let $\delta>0$ be such that \eqref{eq:boundedparam} holds. Moreover, let the sets $E_{M,\varepsilon}$ and $E$ be as in \eqref{eq:egorov}. Relying on the disintegration, it is enough to show that
	$$
	\mathcal{L}^m\times\mu_\hhh^{\pi^1}(A)=0,
	$$
	for $\mu$-almost every $\hhh$. Let $\Omega_\hhh$ be such that $\mu_\hhh^{\pi^1}(\Omega_\hhh)=1$ and $h_{\rm top}(\Omega_\hhh)\leq\alpha+\varepsilon$. That is, for every $\eta>0$ and every $M\geq1$ there exists a countable family $\CC_{\delta,M}$ of cylinders such that $|\iii|\geq N$ for all $\iii\in\CC_{\delta,M}$, $\Omega_\hhh\subset\bigcup_{\iii\in\CC_{\delta,M}}[\iii]$, and $\sum_{\iii\in\CC_{\delta,M}}e^{-|\iii|(\alpha+\varepsilon)}<\eta$. For each $\www\in B(\www_0,\delta)$ define
  \begin{align*}
    A_{\www,n,M}=\{\iii\in\Omega_\hhh\cap E_{\varepsilon,M}:\;&\text{there is }\jjj\in\Omega_\hhh\cap E_{\varepsilon,M}\setminus\{\iii\}\\ 
    &\text{such that }|\jjj\wedge\iii|\leq n\text{ and }\pi_\www^2(\iii)=\pi_\www^2(\jjj)\}
  \end{align*}
	and for every $\iii\in \Omega_\hhh\cap E_{\varepsilon,M}$ let
  \begin{align*}
    A_{\iii,n,M}=\{\www\in B(\www_0,\delta):\;&\text{there is }\jjj\in\Omega_\hhh\cap E_{\varepsilon,M}\setminus\{\iii\}\\ 
    &\text{such that }|\jjj\wedge\iii|\leq n\text{ and }\pi_\www^2(\iii)=\pi_\www^2(\jjj)\}.
  \end{align*}
	If $\pi_\www^2(\iii)=\pi_\www^2(\jjj)$ for some $\jjj\in A_{\www,n,M}$ then, for every $\kkk\in\Sigma_*$ such that $|\kkk|\geq M$ and $\jjj\in[\kkk]$, we have
	\begin{equation*}
		|\pi_\www^2(\iii)-g_{\kkk}^\www(\pi^1(\sigma^{|\kkk|}\jjj),0)|\leq |(g_\kkk^\www)'_y(\pi^1(\sigma^{|\kkk|}\jjj),\xi)|\leq e^{-|\kkk|(\chi_1(\mu,\ttt_0)-2\varepsilon)}.
	\end{equation*}
	for all $\www\in B(\www_0,\delta)$.	Hence, for every $\iii\in\Omega_\hhh\cap E_{\varepsilon,M}$
	$$
	A_{\iii,n,M}\subseteq\bigcup_{\kkk\in\CC_{\delta,M}}\{\www\in B(\www_0,\delta):	|\pi_\www^2(\iii)-g_{\kkk}^\www(\pi^1(\sigma^{|\kkk|}\jjj(\kkk)),0)|\leq  e^{-|\kkk|(\chi_1(\mu,\ttt_0)-2\varepsilon)}\},
	$$
	where $\jjj(\kkk)\in\Omega_\hhh\cap E_{\varepsilon,M}\cap[\kkk]$ is arbitrary.
	
	By the transversality condition \ref{eq:cond5}, there exists $C>0$ such that
	\begin{align*}
		\mathcal{L}^m(\{&\www\in B(\www_0,\delta):	|\pi_\www^2(\iii)-g_{\kkk}^\www(\pi^1(\sigma^{|\kkk|}\jjj(\kkk)),0)|\leq  e^{-|\kkk|(\chi_1(\mu,\ttt_0)-2\varepsilon)}\})\\
		&\qquad\qquad\leq \tau^{-n}\mathcal{L}^m(B(\www_0,\delta))Ce^{-|\kkk|(\chi_1(\mu,\ttt_0)-2\varepsilon)}
	\end{align*}
  for all $\kkk\in\CC_{\delta,M}$, where $\tau<\inf_{x,y}|(g_i)_y'(x,y)|$. Hence,
	\begin{align*}
		\mathcal{L}^m(A_{\iii,n,M})&\leq\tau^{-n}\mathcal{L}^m(B(\www_0,\delta))C\sum_{\kkk\in\CC_{\delta,M}}e^{-|\kkk|(\chi_1(\mu,\ttt_0)-2\varepsilon)}\\ 
    &\leq \tau^{-n}\mathcal{L}^m(B(\www_0,\delta))C\sum_{\kkk\in\CC_{\delta,M}}e^{-|\kkk|(\alpha+\varepsilon)}\\
		&\leq\tau^{-n}\mathcal{L}^m(B(\www_0,\delta))C\eta.
	\end{align*}
	Since $\eta>0$ was arbitrary, we see that $\mathcal{L}^m(A_{\iii,n,M})=0$ for all $n,M\geq1$ and for every $\iii\in\Omega_\hhh\cap E_{\varepsilon,M}$. Thus,
	\begin{equation*}
		\mathcal{L}^m\times\mu_\hhh^{\pi^1}(A)=\mathcal{L}^m\times\mu_\hhh^{\pi^1}(A\cap\Omega_\hhh)=\int\mathcal{L}^m\biggl(\bigcup_{n=1}^\infty\bigcup_{M=1}^\infty A_{\iii,n,M}\biggr)\dd\mu_\hhh^{\pi^1}(\iii)=0
	\end{equation*}
  as wished.
\end{proof}

Let us finish this section by proving Proposition~\ref{prop:main3}.

\begin{proof}[Proof of Proposition~\ref{prop:main3}]
	Fix $\www_0\in V$ and choose $\varepsilon>0$ such that $h(\mu)-\chi_2(\mu)-\chi_1(\mu,\www_0)-3\varepsilon>0$. Let $A_\kkk=\{(\iii,\jjj)\in\Sigma\times\Sigma:\iii\wedge\jjj=\kkk\}$ for all $\kkk\in\Sigma_*$ and $E_{M,\varepsilon}$ be the set defined in \eqref{eq:egorov}. To prove absolute continuity, it is enough to show that for every $\varepsilon>0$ and $M\geq1$ we have
  \begin{equation*}
    (\pi^2_\www)_*\mu_\hhh^{\pi^1}\ll\mathcal{L}
  \end{equation*}
  for $\mu$-almost all $\hhh\in E_{\varepsilon,M}$ and $\mathcal{L}$-almost all $\www\in B(\www_0,\delta)$. To verify this, by \cite[Theorem~2.12]{Mattila1995}, it suffices to show that $\underline{D}((\pi^2_\www)_*\mu_\hhh^{\pi^1},x)<\infty$ for $(\pi^2_\www)_*\mu_\hhh^{\pi^1}|_{E_{M,\varepsilon}}$-almost all $x$, where
	$$
	\underline{D}((\pi^2_\www)_*\mu_\hhh^{\pi^1},x)=\liminf_{r\downarrow 0}\frac{(\pi^2_\www)_*\mu_\hhh^{\pi^1}(B(x,r))}{2r}.
	$$
	Similarly to the proof of Proposition~\ref{prop:main1}, we have
	\begin{align*}
		\mathcal{L}^m(\{&\www\in B(\www_0,\delta):|\pi_\www^2(\iii)-\pi_\www^2(\jjj)|<r\})\\
		&\qquad\qquad\leq \mathcal{L}^m(\{\www\in B(\www_0,\delta):|\pi_\www^2(\sigma^{|\kkk|}\iii)-\pi_\www^2(\sigma^{|\kkk|}\jjj)|<Ce^{|\kkk|(\chi_1(\mu,\www_0)+2\varepsilon)}r\})\\
		&\qquad\qquad\leq C'e^{|\kkk|(\chi_1(\mu,\www_0)+2\varepsilon)}r
	\end{align*}
  for all $r>0$ and $(\iii,\jjj)\in A_{\kkk}\cap \Sigma\times E_{\varepsilon,M}$ with $|\kkk|\geq M$. Thus, by applying Fatou's lemma, we get
	\begin{align*}
		\iint &\underline{D}((\pi^2_\www)_*\mu_\hhh^{\pi^1},x)\dd(\pi^2_\www)_*\mu_\hhh^{\pi^1}|_{E_{\varepsilon,M}}(x)\dd\www\\
		&\leq\liminf_{r\to0}\frac{1}{2r}\iint \mathcal{L}^m(\{\www\in B(\www_0,\delta):\\ 
    &\qquad\qquad\qquad\qquad|\pi_\www^2(\iii)-\pi_\www^2(\jjj)|<r\})\dd\mu_\hhh^{\pi^1}(\iii)\dd\mu_\hhh^{\pi^1}|_{E_{\varepsilon,M}}(\jjj)\\
		&=\liminf_{r\to0}\frac{1}{2r}\sum_{n=0}^\infty\sum_{\kkk\in\Sigma_n}\iint_{A_\kkk} \mathcal{L}^m(\{\www\in B(\www_0,\delta):\\ 
    &\qquad\qquad\qquad\qquad|\pi_\www^2(\iii)-\pi_\www^2(\jjj)|<r\})\dd\mu_\hhh^{\pi^1}(\iii)\dd\mu_\hhh^{\pi^1}|_{E_{\varepsilon,M}}(\jjj)\\
		&\leq\sum_{n=0}^\infty\sum_{\kkk\in\Sigma_n}C'e^{n(\chi_1(\mu,\www_0)+2\varepsilon)}\mu_\hhh^{\pi^1}\times(\mu_\hhh^{\pi^1}|_{E_{\varepsilon,M}})(A_\kkk)\\
		&\leq\sum_{n=0}^\infty\sum_{\kkk\in\Sigma_n}C'e^{n(\chi_1(\mu,\www_0)+2\varepsilon)}e^{-n(h_{\mu}-\chi_2(\mu)-\varepsilon)}\mu_\hhh^{\pi^1}([\kkk])\\
    &\leq\sum_{n=0}^\infty C'e^{-n(h_{\mu}-\chi_2(\mu)-\chi_1(\mu,\www_0)-3\varepsilon)}<\infty,
	\end{align*}
which completes the proof.
\end{proof}

\section{Verifying the examples} \label{sec:verify-examples}

In this section, we verify that the parametrized iterated function systems given in Section~\ref{sec:examples} satisfy the assumptions \ref{eq:cond1}--\ref{eq:cond7}. Theorems \ref{thm:mainex1} and \ref{thm:mainex1b} follow then immediately from Theorem \ref{thm:meta}. Let us first demonstrate that, besides the transversality condition \ref{eq:cond5}, all the other conditions hold almost automatically.

\begin{proposition}\label{prop:ex}
Suppose that $\Phi^\ttt$ satisfies the assumptions \ref{it:condLY1}--\ref{it:condLY2} with $\rho<1/2$ on the open and bounded sets $V,W\subset\R^{N}$ such that $\vvv=(t_{1,1},\ldots,t_{N,1})\in V$ and $\www=(t_{1,2},\ldots,t_{N,2})\in W$. Then $\Phi^{\ttt}$ satisfies \ref{eq:cond1}--\ref{eq:cond4} and \ref{eq:cond6}--\ref{eq:cond7}.
\end{proposition}

\begin{proof}
	The conditions \ref{eq:cond1}-\ref{eq:cond3} hold trivially by the compactness of $[0,1]^2$. Furthermore, \ref{eq:cond6} follows by Simon and Solomyak \cite[Lemma~3.1]{SiSo} and \ref{eq:cond7} follows by \cite[Lemma~3.3]{SiSo}. The proof of \ref{eq:cond4} is similar to the proof of \ref{eq:cond6} but we give the details for completeness. Observe that
	$$
	  |f_{\iii}^{\vvv}(x)-f_{\iii}^{\vvv_0}(x)| \leq \frac{|\vvv-\vvv_0|}{1-\rho}
	$$
  for all $\iii\in\Sigma_*$, $x\in[0,1]$, and $\vvv,\vvv_0\in V$. Furthermore, for every $(x,y)\in[0,1]^2$ and $\ttt,\ttt_0\in V\times W$, we have
	\begin{align*}
	|g_{\iii}^{\ttt}(x,y)-g_{\iii}^{\ttt_0}(x,y)| &\leq |t_{i_1,2}-t_{i_1,2}'|+\sup|(g_{i_1})'_x||f_{\sigma\iii}^{\vvv}(x)-f_{\sigma\iii}^{\vvv_0}(x)|\\ 
  &\qquad\qquad+\sup|(g_{i_1})'_y||g_{\sigma\iii}^{\ttt}(x,y)-g_{\sigma\iii}^{\ttt_0}(x,y)|
	\end{align*}
	and hence, by the induction,
	$$
    |g_{\iii}^{\ttt}(x,y)-g_{\iii}^{\ttt_0}(x,y)| \leq \frac{|\www-\www_0|(1-\rho)+C|\vvv-\vvv_0|}{(1-\rho)^2}.
	$$
	Therefore,
	\begin{align*}
	\log\frac{|(g_{\iii}^{\ttt})'_y(x,y)|}{|(g_{\iii}^{\ttt_0})'_y(x,y)|}&=\sum_{k=1}^{|\iii|}\log\frac{|(g_{i_k})'_y(F_{\sigma^k\iii}^{\ttt}(x,y))|}{|(g_{i_k})'_y(F_{\sigma^k\iii}^{\ttt_0}(x,y))|}\\
	&=\sum_{k=1}^{|\iii|}\log\frac{|(g_{i_k})'_y(f_{\sigma^k\iii}^{\vvv}(x),g_{\sigma^k\iii}^{\ttt}(x,y))|}{|(g_{i_k})'_y(f_{\sigma^k\iii}^{\vvv}(x),g_{\sigma^k\iii}^{\ttt_0}(x,y))|}\\ 
  &\qquad\qquad+\log\frac{|(g_{i_k})'_y(f_{\sigma^k\iii}^{\vvv}(x),g_{\sigma^k\iii}^{\ttt_0}(x,y))|}{|(g_{i_k})'_y(f_{\sigma^k\iii}^{\vvv_0}(x),g_{\sigma^k\iii}^{\ttt_0}(x,y))|}\\	&=\sum_{k=1}^{|\iii|}\frac{|(g_{i_k})''_{yy}(f_{\sigma^k\iii}^{\vvv}(x),\xi)|}{|(g_{i_k})'_y(f_{\sigma^k\iii}^{\vvv}(x),\xi)|}|g_{\sigma^k\iii}^{\ttt}(x,y)-g_{\sigma^k\iii}^{\ttt_0}(x,y)|\\
	&\qquad\qquad+\sum_{k=1}^{|\iii|}\frac{|(g_{i_k})''_{yx}(\zeta,g_{\sigma^k\iii}^{\ttt_0}(x,y))|}{|(g_{i_k})'_y(\zeta,g_{\sigma^k\iii}^{\ttt_0}(x,y))|}|f_{\sigma^k\iii}^{\vvv}(x)-f_{\sigma^k\iii}^{\vvv_0}(x)|\\
	&\leq |\iii|C'(|\vvv-\vvv_0|+|\www-\www_0|),
	\end{align*}
which had to be proven.
\end{proof}

In view of Proposition~\ref{prop:ex}, it is enough to verify that both the conditions \ref{it:1}--\ref{it:3} and \ref{it:1b}--\ref{it:3b} imply the transversality condition \ref{eq:cond5}. The following general lemma highlights that to show the transversality, it suffices to study derivatives. The proof we present is standard, and is similar to \cite[proof of Lemma~3.3]{SiSo}.

\begin{lemma}\label{lem:techtrans}
	Suppose that $X$ is a compact metric space and $f\colon U\times X\to\R$ is a continuous map such that $\ttt\mapsto f(\ttt,x)$ is continuously differentiable on an open and bounded set $U \subset \R^d$. If for every $x\in X$ there exists $i\in\{1,\ldots,d\}$ such that 
	$$
    f(\ttt_0,x)=0 \quad\Rightarrow\quad \frac{\partial}{\partial t_i}f(\ttt,x)\Big|_{\ttt=\ttt_0}>0,
	$$
	then for every compact set $K\subset U$ there exist $C>0$ such that
	$$
    \mathcal{L}^d(\{\ttt\in K:|f(\ttt,x)|\leq r\})\leq Cr
	$$
  for all $x\in X$ and all $r>0$.
\end{lemma}

\begin{proof}
	Let $K\subset U$ be a compact set. Let us first show that there exists $\delta>0$ such that for every $x\in X$ and $\ttt\in K$ there is $i\in\{1,\ldots,d\}$ with
	\begin{equation}\label{eq:deriv}
    |f(\ttt,x)|<\delta \quad\Rightarrow\quad \Bigl|\frac{\partial}{\partial t_i}f(\ttt,x)\Bigr| \geq \delta.
	\end{equation}
	Let us argue by contradiction. That is, for every $n\in\N$ there exists $x_n\in\N$ and $\ttt_n\in K$ such that for every $i\in\{1,\ldots,d\}$
	$$
    |f(\ttt_n,x_n)| \leq \frac{1}{n} \qquad\text{and}\qquad \Bigl|\frac{\partial}{\partial t_i}f(\ttt_n,x_n)\Bigr|<\frac{1}{n}.
	$$
	By using the compactness and continuity, there exists a sequence $(n_k)_k$, $x\in X$, and $\ttt\in K$ such that $x = \lim_{k\to\infty}x_{n_k}$, $\ttt = \lim_{k\to\infty}\ttt_{n_k}$, and
	$$
    |f(\ttt,x)| = \Bigl|\frac{\partial}{\partial t_i}f(\ttt,x)\Bigr|=0
	$$
  for all $i\in\{1,\ldots,d\}$, which contradicts our assumption.
	
	Fix $x\in X$ and for every $i\in\{1,\ldots,d\}$, let $Q_i\subset K$ be the closed subset for which \eqref{eq:deriv} holds for $i$. Let us define a map $T_i\colon\R^d\to\R^d$ by setting
	$$
    T_i(t_1,\ldots,t_d)=(t_1,\ldots,t_{i-1},f(t_1,\ldots,t_d,x),t_{i+1},\ldots,t_d)
	$$
	Furthermore, let $A_i=\R^{i-1}\times[-r,r]\times\R^{d-i}$ and $\Pi^i\colon\R^d\mapsto\R^{d-1}$ be the orthogonal projection along the $i$th coordinate axis. By \eqref{eq:deriv}, $T_i$ is a smooth and invertible map on $\{\ttt\in K:|f(\ttt,x)|<\delta\}$ as $|\det(D_\ttt T_i)|=|\frac{\partial}{\partial t_i}f(\ttt,x)|\geq\delta$. Therefore,
	\begin{align*}
	\mathcal{L}^d(\{\ttt\in Q_i:|f(\ttt,x)|<r\}) &= \mathcal{L}^d(T_i^{-1}(A_i)\cap Q_i)\\
	&\leq\max_{\ttt\in Q_i}|\det(D_\ttt T_i^{-1})|\mathcal{L}^d(A_i\cap T_i(K))\\
	&\leq\delta^{-1}\mathcal{L}^{d-1}(\Pi^i(T_i(K)))2r\\
	\end{align*}
  for all $r<\delta$ and
  \begin{equation*}
    \mathcal{L}^d(\{\ttt\in K:|f(\ttt,x)|<r\})\leq\mathcal{L}^d(K)\delta^{-1}r
  \end{equation*}
  for all $r\geq\delta$. This completes the proof.
\end{proof}

Throughout the remaining part of the section, we fix $\vvv_0=(t_{1,1},\ldots,t_{N,1})$ and, as in Section \ref{sec:proof-props}, we denote $\pi^1_{\vvv_0}$ simply by $\pi^1$.

\subsection{The first example}

In this subsection, we assume that $\Phi^\ttt$ satisfies the assumptions \ref{it:1}--\ref{it:3} and prove the transversality condition \ref{eq:cond5}. We will see that it follows from the following proposition.

\begin{proposition}\label{lem:basictrans}
  Suppose that $\Phi^\ttt$ satisfies the assumptions \ref{it:1}--\ref{it:3}. Then there exists $\delta>0$ such that for every $\iii,\jjj,\hhh\in\Sigma$ with $j_1\neq h_1$ and $\pi^1(\jjj)\leq\pi^1(\hhh)$ the following holds: if $y_\ttt(\iii,\pi_\ttt(\jjj),x)\equiv y_\ttt(\iii,\pi_\ttt(\hhh),x)$ as the function of $x$, then
	$$
	\frac{\partial}{\partial t_{j_1,2}}(y_\ttt(\iii,\pi_\ttt(\jjj),x)-y_\ttt(\iii,\pi_\ttt(\hhh),x))>\delta
	$$
	for all $x\in[0,1]$ and $\ttt=(\vvv_0,\www)\in V\times W$.
\end{proposition}

Indeed, the transversality condition \ref{eq:cond5} follows from Proposition~\ref{lem:basictrans} by applying Lemma~\ref{lem:techtrans} for the map $f(\ttt,(\jjj,\hhh,x))=y_\ttt(\iii,\pi_\ttt(\jjj),x)-y_\ttt(\iii,\pi_\ttt(\hhh),x)$ together with the fact given by Lemma~\ref{lem:dist} that there exists a constant $C>0$ such that
$$
  \mathcal{L}^{2N}(\{\ttt:|\proj_\iii^{\ttt}(\jjj)-\proj_\iii^{\ttt}(\hhh)|\leq r\})\leq\mathcal{L}^{2N}(\{\ttt:|y_{\ttt}(\iii,\pi_\ttt(\jjj),x)-y_{\ttt}(\iii,\pi_\ttt(\hhh),x)|\leq Cr\})
$$
for all $x\in[0,1]$, $\iii\in\Sigma$, and $\jjj,\hhh\in\Sigma$ with $j_1\neq h_1$.

Let us begin preparations for the proof of Proposition \ref{lem:basictrans}. At first, we observe that
\begin{equation}\label{eq:finiteder}
\frac{\partial}{\partial t_{j,2}}g_{\iii}^{\ttt}(x,y)\geq0
\end{equation}
for all $\iii\in\Sigma_*$, $j \in \{1,\ldots,N\}$, and $\ttt\in V\times W$. In particular,
\begin{equation}\label{eq:natprojder}
	0\leq\frac{\partial}{\partial t_{j,2}}\pi_\ttt^2(\jjj)\leq\frac{1}{1-\rho}
\end{equation}
for all $\jjj\in\Sigma$, $j \in \{1,\ldots,N\}$, and $\ttt\in V\times W$. Indeed,
\begin{equation}\label{eq:transinter}
  \frac{\partial}{\partial t_{j,2}}\pi_\ttt^2(\jjj)=\delta_{j_1}^j+(g_{j_1})'_y(\pi_\ttt(\sigma\jjj))\frac{\partial}{\partial t_{j,2}}\pi_\ttt^2(\sigma\jjj)=\sum_{k=1}^\infty\delta_{j_k}^j(g_{\jjj|_{k-1}}^\ttt)'_y(\pi_\ttt(\sigma^k\jjj)),
\end{equation}
where $\delta_i^j=1$ if $i=j$ and $0$ otherwise. Hence, \eqref{eq:natprojder} follows from \ref{it:1}. Similarly, \eqref{eq:finiteder} follows from \ref{it:1} and \eqref{eq:transinter} but with considering finite sums. Furthermore, it follows from \eqref{eq:transinter} that for every $\jjj,\hhh\in\Sigma$ with $h_1\neq j_1$
\begin{equation}\label{eq:usualtrans}
\frac{\partial}{\partial t_{j_1,2}}(\pi_\ttt^2(\jjj)-\pi_\ttt^2(\hhh))\geq1-\frac{\rho}{1-\rho}>0
\end{equation}
for all $\ttt \in V\times W$.

\begin{lemma}\label{lem:posu}
  Suppose that $\Phi^\ttt$ satisfies the assumptions \ref{it:1} and \ref{it:2}. Then
  \begin{equation*} 
    \frac{\partial}{\partial t_{j,2}}u_\ttt(\iii,x,y) \ge 0 \qquad\text{and}\qquad (u_\ttt)'_y(\iii,x,y) \ge 0
  \end{equation*}
  for all $\iii\in\Sigma$, $j \in \{1,\ldots,N\}$, $(x,y)\in[0,1]^2$, and $\ttt=(\vvv_0,\www)\in V\times W$. 
\end{lemma}

\begin{proof}
Simple calculations show that
\begin{align}
		\frac{\partial}{\partial t_{j,2}}u_\ttt(\iii,x,y) &= \sum_{k=1}^\infty\biggl(\frac{-(g_{i_k})_{xy}''(F_{\overleftarrow{\iii|_{k-1}}}^\ttt(x,y))f_{\overleftarrow{\iii|_{k-1}}}'(x)}{(g_{\overleftarrow{\iii|_{k}}}^\ttt)_y'(x,y)}\frac{\partial}{\partial t_{j,2}}g_{\overleftarrow{\iii|_{k-1}}}^\ttt(x,y) \label{eq:derofu} \\
		&+\sum_{\ell=1}^k\frac{(g_{i_k})_x'(F_{\overleftarrow{\iii|_{k-1}}}^\ttt(x,y))f_{\overleftarrow{\iii|_{k-1}}}'(x)}{(g_{\overleftarrow{\iii|_{k}}}^\ttt)_y'(x,y)}\frac{(g_{i_\ell})_{yy}''(F_{\overleftarrow{\iii|_{\ell-1}}}^\ttt(x,y))}{(g_{i_\ell})_{y}'(F_{\overleftarrow{\iii|_{\ell-1}}}^\ttt(x,y))}\frac{\partial}{\partial t_{j,2}}g_{\overleftarrow{\iii|_{\ell-1}}}^\ttt(x,y)\biggr), \notag
\end{align}
where the series converges by \ref{it:1}. Hence, the first claim follows from \ref{it:2} and \eqref{eq:finiteder}. The second claim follows from \eqref{eq:uposy} by a similar manner.
\end{proof}

Let us next write the ordinary differential equation \eqref{eq:diffeq} in the integral equation form. That is, for every $\iii,\jjj\in\Sigma$, $\ttt\in V\times W$, and $x\in[0,1]$, we have
$$
y_\ttt(\iii,\pi_\ttt(\jjj),x)=\pi_\ttt^2(\jjj)+\int_{\pi^1(\jjj)}^xu_\ttt(\iii,z,y_\ttt(\iii,\pi_\ttt(\jjj),z))\dd z.
$$
Thus,
\begin{align*}
	\frac{\partial}{\partial t_{h,2}}y_\ttt(\iii,\pi_\ttt(\jjj),x)&=\frac{\partial}{\partial t_{h,2}}\pi_\ttt^2(\jjj)+
	\int_{\pi^1(\jjj)}^x\biggl(\frac{\partial}{\partial t_{h,2}}u_\ttt\biggr)(\iii,z,y_\ttt(\iii,\pi_\ttt(\jjj),z))\\ 
  &\qquad\qquad+(u_\ttt)'_y(\iii,z,y_\ttt(\iii,\pi_\ttt(\jjj),z))\frac{\partial}{\partial t_{h,2}}y_\ttt(\iii,\pi_\ttt(\jjj),z)\dd z
\end{align*}
for all $h\in\{1,\ldots,N\}$. To simplify notation, we write $y(x)=y_\ttt(\iii,\pi_\ttt(\jjj),x)$. Recalling how to solve linear nonhomogeneous ordinary differential equations (see, for example, \cite[Section~2.3]{ODE}), we can write
\begin{equation}\label{eq:solution}
\begin{split}
	\frac{\partial}{\partial t_{h,2}}y_\ttt(\iii,&\pi_\ttt(\jjj),x)=\frac{\partial}{\partial t_{h,2}}\pi_\ttt^2(\jjj)\exp\biggl(\int_{\pi^1(\jjj)}^x(u_\ttt)'_y(\iii,z,y(z))\dd z\biggr)\\ 
  &+\int_{\pi^1(\jjj)}^x\exp\biggl(\int_w^x(u_\ttt)'_y(\iii,z,y(z))\dd z\biggr)\biggl(\frac{\partial}{\partial t_{h,2}}u_\ttt\biggr)(\iii,w,y(w))\dd w.
\end{split}
\end{equation}
We are now ready to prove Proposition \ref{lem:basictrans}.

\begin{proof}[Proof of Proposition~\ref{lem:basictrans}]
To simplify notation, write $y(x)=y_\ttt(\iii,\pi_\ttt(\jjj),x)\equiv y_\ttt(\iii,\pi_\ttt(\hhh),x)$. Applying \eqref{eq:solution} for both $\frac{\partial}{\partial t_{j_1,2}}y_\ttt(\iii,\pi_\ttt(\jjj),x)$ and $\frac{\partial}{\partial t_{j_1,2}}y_\ttt(\iii,\pi_\ttt(\hhh),x)$, we get
\begin{align*}
	\frac{\partial}{\partial t_{j_1,2}}&y_\ttt(\iii,\pi_\ttt(\jjj),x)-\frac{\partial}{\partial t_{j_1,2}}y_\ttt(\iii,\pi_\ttt(\hhh),x)\\
	&=\frac{\partial}{\partial t_{j_1,2}}\pi_\ttt^2(\jjj)\exp\biggl(\int_{\pi^1(\jjj)}^x(u_\ttt)'_y(\iii,z,y(z))\dd z\biggr)\\ 
  &\qquad\qquad-\frac{\partial}{\partial t_{j_1,2}}\pi_\ttt^2(\hhh)\exp\biggl(\int_{\pi^1(\hhh)}^x(u_\ttt)'_y(\iii,z,y(z))\dd z\biggr)\\
	&\qquad\qquad+\int_{\pi^1(\jjj)}^{\pi^1(\hhh)}\exp\biggl(\int_w^x(u_\ttt)'_y(\iii,z,y(z))\dd z\biggr)\biggl(\frac{\partial}{\partial t_{j_1,2}}u_\ttt\biggr)(\iii,w,y(w))\dd w.
\end{align*}
Since
\begin{align*}
  -\int_{\pi^1(\jjj)}^{\pi^1(\hhh)}&\exp\biggl(\int_w^x(u_\ttt)'_y(\iii,z,y(z))\dd z\biggr)(u_\ttt)'_y(\iii,w,y(w))\dd w\\ 
  &=\exp\biggl(\int_{\pi^1(\hhh)}^x(u_\ttt)'_y(\iii,z,y(z))\dd z\biggr)-\exp\biggl(\int_{\pi^1(\jjj)}^x(u_\ttt)'_y(\iii,z,y(z))\dd z\biggr),
\end{align*}
we get
\begin{align*}
	\frac{\partial}{\partial t_{j_1,2}}&y_\ttt(\iii,\pi_\ttt(\jjj),x)-\frac{\partial}{\partial t_{j_1,2}}y_\ttt(\iii,\pi_\ttt(\hhh),x)\\
	&=\biggl(\frac{\partial}{\partial t_{j_1,2}}\pi_\ttt^2(\jjj)-\frac{\partial}{\partial t_{j_1,2}}\pi_\ttt^2(\hhh)\biggr)\exp\biggl(\int_{\pi^1(\hhh)}^x(u_\ttt)'_y(\iii,z,y(z))\dd z\biggr)\\
	&\qquad\qquad+\int_{\pi^1(\jjj)}^{\pi^1(\hhh)}\exp\biggl(\int_w^x(u_\ttt)'_y(\iii,z,y(z))\dd z\biggr)\biggl(\biggl(\frac{\partial}{\partial t_{j_1,2}}u_\ttt\biggr)(\iii,w,y(w))\\ 
  &\qquad\qquad\qquad\qquad\qquad\qquad\qquad\qquad+\frac{\partial}{\partial t_{j_1,2}}\pi_\ttt^2(\jjj)(u_\ttt)'_y(\iii,w,y(w))\biggr)\dd w.
\end{align*}
Therefore, by \eqref{eq:ubounded}, \eqref{eq:usualtrans}, and Lemma~\ref{lem:posu}, there exists $C\geq0$ such that
$$
\frac{\partial}{\partial t_{j_1,2}}y_\ttt(\iii,\pi_\ttt(\jjj),x)-\frac{\partial}{\partial t_{j_1,2}}y_\ttt(\iii,\pi_\ttt(\hhh),x)\geq\biggl(1-\frac{\rho}{1-\rho}\biggr)e^{-C}>0
$$
which is what we wanted to prove.
\end{proof}

\subsection{The second example}

In this subsection, we assume that $\Phi^\ttt$ satisfies the assumptions \ref{it:1b}--\ref{it:3b} and prove the transversality condition \ref{eq:cond5}. As with the first example, we will see that it follows from the following proposition.

\begin{proposition}\label{lem:basictransv2}
  Suppose that $\Phi^\ttt$ satisfies the assumptions \ref{it:1b}--\ref{it:3b}. Then there exists $\delta>0$ such that for every $\iii,\jjj,\hhh\in\Sigma$ with $j_1\neq h_1$ and $\pi^1(\jjj)\leq\pi^1(\hhh)$ the following holds: if $y_\ttt(\iii,\pi_\ttt(\jjj),x)\equiv y_\ttt(\iii,\pi_\ttt(\hhh),x)$ as a function of $x$, then there exists $k\in\{1,\ldots,N\}$ such that
	$$
	\biggl|\frac{\partial}{\partial t_{k,2}}(y_\ttt(\iii,\pi_\ttt(\jjj),x)-y_\ttt(\iii,\pi_\ttt(\hhh),x))\biggr|>\delta,
	$$
	where $x=(\pi^1(\hhh)+\pi^1(\jjj))/2$.
\end{proposition}

Before we prove the proposition, we need some more delicate analysis on the bounds of the derivatives on the tangent bundle $u_\ttt(\iii,x,y)$.

\begin{lemma}\label{lem:uyandut}
  Under the assumptions of Proposition~\ref{lem:basictransv2},
	$$
	|(u_\ttt)'_y(\iii,x,y)|\leq\frac{112}{135}
	$$
  for all $\iii\in\Sigma$, $\ttt=(\vvv_0,\www)\in V\times W$, and $(x,y)\in[0,1]^2$. Furthermore,
	\begin{equation}\label{eq:bound2b}
		\biggl|\frac{\partial}{\partial t_{j,2}}u_\ttt(\iii,x,y)-\frac{\partial}{\partial t_{h,2}}u_\ttt(\iii,x,y)\biggr|\le\frac{28}{81}.
	\end{equation}
  for all $j,h\in\{1,\ldots,N\}$ with $j\neq h$.
\end{lemma}

\begin{proof}
	The first claim follows from \eqref{eq:uposy} and the assumptions, since
	$$
	|(u_\ttt)'_y(\iii,x,y)|\leq\sum_{k=1}^\infty\Bigl(\frac13\Bigl(\frac14\Bigr)^{k-1}\Bigl(\frac14\Bigr)^{k-1}+\sum_{\ell=1}^k\frac13\Bigl(\frac14\Bigr)^{k-1}\Bigl(\frac14\Bigr)^{\ell-1}\Bigr)=\frac{112}{135}.
	$$
	Let us show the second claim. By using \eqref{eq:transinter}, it is easy to see that
	$$
	\biggl|\frac{\partial}{\partial t_{j,2}}g_{\iii}^\ttt(x,y)-\frac{\partial}{\partial t_{h,2}}g_{\iii}^\ttt(x,y)\biggr|\leq\sum_{k=0}^{|\iii|-1}\Bigl(\frac14\Bigr)^{k}\leq\frac43
	$$
  for all $\iii\in\Sigma_*$. Observe that in both sums in \eqref{eq:derofu}, the first terms are zero, and so 
	\begin{equation*}
		\begin{split}
			\frac{\partial}{\partial t_{j,2}}&u_\ttt(\iii,x,y)-\frac{\partial}{\partial t_{h,2}}u_\ttt(\iii,x,y)\\
			&=\sum_{k=2}^\infty\biggl(\frac{-(g_{i_k})_{xy}''(F_{\overleftarrow{\iii|_{k-1}}}^\ttt(x,y))f_{\overleftarrow{\iii|_{k-1}}}'(x)}{(g_{\overleftarrow{\iii|_{k}}}^\ttt)_y'(x,y)}\biggl(\frac{\partial}{\partial t_{j,2}}g_{\overleftarrow{\iii|_{k-1}}}^\ttt(x,y)-\frac{\partial}{\partial t_{h,2}}g_{\overleftarrow{\iii|_{k-1}}}^\ttt(x,y)\biggr)\\
			&\qquad\qquad+\sum_{\ell=2}^k\frac{(g_{i_k})_x'(F_{\overleftarrow{\iii|_{k-1}}}^\ttt(x,y))f_{\overleftarrow{\iii|_{k-1}}}'(x)}{(g_{\overleftarrow{\iii|_{k}}}^\ttt)_y'(x,y)}\frac{(g_{i_\ell})_{yy}''(F_{\overleftarrow{\iii|_{\ell-1}}}^\ttt(x,y))}{(g_{i_\ell})_{y}'(F_{\overleftarrow{\iii|_{\ell-1}}}^\ttt(x,y))}\\ 
      &\qquad\qquad\qquad\qquad\qquad\qquad\cdot\biggl(\frac{\partial}{\partial t_{j,2}}g_{\overleftarrow{\iii|_{\ell-1}}}^\ttt(x,y)-\frac{\partial}{\partial t_{h,2}}g_{\overleftarrow{\iii|_{\ell-1}}}^\ttt(x,y)\biggr)\biggr)\\
			&=\sum_{k=2}^\infty\biggl(\frac{\partial}{\partial t_{j,2}}g_{\overleftarrow{\iii|_{k-1}}}^\ttt(x,y)-\frac{\partial}{\partial t_{h,2}}g_{\overleftarrow{\iii|_{k-1}}}^\ttt(x,y)\biggr)\biggl(\frac{-(g_{i_k})_{xy}''(F_{\overleftarrow{\iii|_{k-1}}}^\ttt(x,y))f_{\overleftarrow{\iii|_{k-1}}}'(x)}{(g_{\overleftarrow{\iii|_{k}}}^\ttt)_y'(x,y)}\\
			&\qquad\qquad+\frac{(g_{i_k})_{yy}''(F_{\overleftarrow{\iii|_{k-1}}}^\ttt(x,y))}{(g_{i_k})_{y}'(F_{\overleftarrow{\iii|_{k-1}}}^\ttt(x,y))}\sum_{\ell=k}^\infty\frac{(g_{i_\ell})_x'(F_{\overleftarrow{\iii|_{\ell-1}}}^\ttt(x,y))f_{\overleftarrow{\iii|_{\ell-1}}}'(x)}{(g_{\overleftarrow{\iii|_{\ell}}}^\ttt)_y'(x,y)}\biggr).
		\end{split}
	\end{equation*}
Thus, by the assumptions,
\[
\biggl|\frac{\partial}{\partial t_{j,2}}u_\ttt(\iii,x,y)-\frac{\partial}{\partial t_{h,2}}u_\ttt(\iii,x,y)\biggr|\leq\sum_{k=2}^\infty\frac43\Bigl(\frac13\cdot\Bigl(\frac14\Bigr)^{k-1}+\frac13\cdot\sum_{\ell=k}^\infty\Bigl(\frac14\Bigr)^{\ell-1}\Bigr)=\frac{28}{81}
\]
as claimed.
\end{proof}

\begin{proof}[Proof of Proposition~\ref{lem:basictransv2}]
To simplify notation, let $y(x) = y_\ttt(\iii,\pi_\ttt(\jjj),x)\equiv y_\ttt(\iii,\pi_\ttt(\hhh),x)$. By applying \eqref{eq:solution}, we get
\begin{align*}
		\biggl(\frac{\partial}{\partial t_{j_1,2}}&y_\ttt(\iii,\pi_\ttt(\jjj),x)-\frac{\partial}{\partial t_{j_1,2}}y_\ttt(\iii,\pi_\ttt(\hhh),x)\biggr)\\ 
    &\qquad\qquad-\biggl(\frac{\partial}{\partial t_{h_1,2}}y_\ttt(\iii,\pi_\ttt(\jjj),x)-\frac{\partial}{\partial t_{h_1,2}}y_\ttt(\iii,\pi_\ttt(\hhh),x)\biggr)\\
		&=\biggl(\frac{\partial}{\partial t_{j_1,2}}\pi_\ttt^2(\jjj)-\frac{\partial}{\partial t_{h_1,2}}\pi_\ttt^2(\jjj)\biggr)\exp\biggl(\int_{\pi^1(\jjj)}^x(u_\ttt)'_y(\iii,z,y(z))\dd z\biggr)\\
		&\qquad\qquad+\biggl(\frac{\partial}{\partial t_{h_1,2}}\pi_\ttt^2(\hhh)-\frac{\partial}{\partial t_{j_1,2}}\pi_\ttt^2(\hhh)\biggr)\exp\biggl(\int_{\pi^1(\hhh)}^x(u_\ttt)'_y(\iii,z,y(z))\dd z\biggr)\\
		&\qquad\qquad+\int_{\pi^1(\jjj)}^{\pi^1(\hhh)}\exp\biggl(\int_w^x(u_\ttt)'_y(\iii,z,y(z))\dd z\biggr)\biggl(\biggl(\frac{\partial}{\partial t_{j_1,2}}u_\ttt\biggr)(\iii,w,y(w))\\ 
    &\qquad\qquad\qquad\qquad\qquad\qquad\qquad-\biggl(\frac{\partial}{\partial t_{h_1,2}}u_\ttt\biggr)(\iii,w,y(w))\biggr)\dd w.
\intertext{By \eqref{eq:usualtrans} and Lemma~\ref{lem:uyandut}, we can estimate the above from below by}
		&\geq\frac23\exp\Bigl(-\frac{112}{135}(x-\pi^1(\jjj))\Bigr)+\frac23\exp\Bigl(-\frac{112}{135}(\pi^1(\hhh)-x)\Bigr)\\ 
    &\qquad\qquad-\frac{56}{81}\int_{\pi^1(\iii)}^{\pi^1(\hhh)}\exp\Bigl(\frac{112}{135}|x-w|\Bigr)\dd w\\
		&=\frac23\exp\Bigl(-\frac{112}{135}(x-\pi^1(\jjj))\Bigr)+\frac23\exp\Bigl(-\frac{112}{135}(\pi^1(\hhh)-x)\Bigr)\\ 
    &\qquad\qquad-\frac{5}{6}\Bigl(\exp\Bigl(\frac{112}{135}(x-\pi^1(\jjj))\Bigr)+\exp\Bigl(\frac{112}{135}(\pi^1(\hhh)-x)\Bigr)-2\Bigr).
\intertext{Writing $x=(\pi^1(\jjj)+\pi^1(\hhh))/2$ and $z=x-\pi^1(\jjj)=\pi^1(\hhh)-x$, we can continue the estimation by}
		&=\Bigl(\frac43-\frac56\exp\Bigl(\frac{224}{135}z\Bigr)+\frac53\exp\Bigl(\frac{112}{135}z\Bigr)\Bigr)\exp\Bigl(-\frac{112}{135}z\Bigr)\\ 
    &\qquad\qquad\geq\Bigl(\frac43-\frac56\exp\Bigl(\frac{224}{135}\Bigr)+\frac53\exp\Bigl(\frac{112}{135}\Bigr)\Bigr)e^{-\frac{112}{135}}>0,
\end{align*}
where the last two inequalities follow by simple calculus. The claim then holds for either $k=j_1$ or $k=h_1$ with the choice $\delta=(\frac23-\frac{5}{12}\exp(\frac{224}{135})+\frac56\exp(\frac{112}{135}))\exp(-\frac{112}{135})$.
\end{proof}

\bibliographystyle{abbrv}
\bibliography{Bibliography}

\end{document}